\documentclass[a4paper,oneside,10pt]{article}
\usepackage[english]{babel}
\usepackage[T1]{fontenc}
\usepackage[a4paper,top=3cm,bottom=3cm,left=3cm,right=3cm,marginparwidth=1.75cm]{geometry}
\usepackage[colorlinks=true, linkcolor=blue, citecolor=teal, hypertexnames=false]{hyperref}
\usepackage{empheq}
\usepackage{graphicx}
\usepackage[colorinlistoftodos]{todonotes}
\usepackage{amsmath}
\usepackage{mathtools}
\usepackage{upgreek}
\usepackage{amssymb}
\usepackage{amsfonts}
\usepackage{amsthm}
\usepackage{hyperref}
\usepackage{enumitem}
\usepackage{mathrsfs}
\usepackage{fontenc}
\usepackage{verbatim}
\usepackage{framed}
\usepackage{algorithm}
\usepackage{bbm}
\usepackage{algpseudocode}
\usepackage{appendix}
\usepackage{color}
\usepackage{multirow}
\usepackage{lipsum}
\usepackage{authblk}
\usepackage{geometry}
\providecommand{\U}[1]{\protect\rule{.1in}{.1in}}

\newtheorem{theorem}{Theorem}[section]

\newtheorem{corollary}[theorem]{Corollary}

\newtheorem{definition}[theorem]{Definition}
\newtheorem{assumption}[theorem]{Assumption}
\newtheorem{example}[theorem]{Example}

\newtheorem{lemma}[theorem]{Lemma}

\newtheorem{proposition}[theorem]{Proposition}
\newtheorem{remark}[theorem]{Remark}

\numberwithin{equation}{section}

\newcommand{\corresponding}{\footnote{Corresponding author.}}

\title{A BSDE approach to the asymmetric risk-sensitive optimization and its applications
\thanks{All the auhtors contributed equally to this work.}
}

\author[1]{Mingshang Hu}

\author[1]{Shaolin Ji\corresponding}

\author[2]{Rundong Xu}

\author[3]{Xiaole Xue}

\affil[1]{\textit{Zhongtai Securities Institute for Financial Studies, Shandong University, Jinan 250100, China. jsl@sdu.edu.cn}}

\affil[2]{\textit{Center for Applied Mathematics, Tianjin University, Tianjin 300072, China. rundong.xu@mail.sdu.edu.cn}}

\affil[3]{\textit{School of Management, Shandong University, Jinan 250100, China. xlxue@sdu.edu.cn}}

\date{}

\begin{document}

\maketitle


\textbf{Abstract}. This paper is devoted to proposing a new asymmetric
risk-sensitive criterion involving different risk attitudes toward varying
risk sources. The criterion can only be defined through the initial value of 
the minimal solutions of quadratic backward stochastic differential equations (BSDEs). 
Before uncovering the mean-variance representation for the introduced
criterion by the variational approach, some axioms are given for the first
time to characterize a variance decomposition of square integrable random variables.
The stochastic control problems under this criterion are described as 
a kind of stochastic recursive control problems that includes controlled quadratic BSDEs.
An asymmetric risk-sensitive global stochastic maximum principle is derived when the 
quadratic BSDEs are equipped with bounded data.
A closed-form solution of a stochastic linear-quadratic risk-sensitive control problem 
is obtained by introducing a novel completion-of-squares technique for controlled quadratic BSDEs.
In addition, a dynamic portfolio optimization problem featuring a stochastic return
rate is provided as an application of the asymmetric risk-sensitive control.

{\textbf{Key words}. } asymmetric risk-sensitive criterion; quadratic
backward stochastic differential equation (quadratic BSDE); 
linear-quadratic (LQ) optimal control; recursive utility

\textbf{AMS subject classifications.} 93E20, 60H10, 49K45

\addcontentsline{toc}{section}{\hspace*{1.8em}Abstract}

\section{Introduction}

In finance and economics, not all behaviors can be described by risk-neutral
cost functions. One way of capturing risk-sensitivity (including risk-seeking
and risk-averse behavior) is replacing the linear expectation with the
following nonlinear one, for a random variable $\xi$ and a constant $\theta$,
considering
\begin{equation}
\mathcal{E}_{\theta}[\xi]:=\frac{1}{\theta}\log\mathbb{E}[e^{\theta\xi}],
\label{intro-risk-cost}%
\end{equation}
which is the well-known risk-sensitive criterion \cite{Kuroda-Nagai}. It
is obvious that $\mathcal{E}_{\theta}$ is a nonlinear expectation viewed as an
operator preserving monotonicity and constants (see \cite{Coquet} and
references therein). Performing the second-order Taylor expansion of
$\mathcal{E}_{\theta}[\xi]$ with respect to $\theta$ around $\theta=0$, the
criterion (\ref{intro-risk-cost}) is approximated in the following way
\begin{equation}
\mathcal{E}_{\theta}[\xi]=\mathbb{E}\left[  \xi\right]  +\frac{\theta}%
{2}\mathrm{Var}\left[  \xi\right]  +O(\theta^{2}), \label{intro-expan}%
\end{equation}
where $\mathrm{Var}[\xi]$ is the variance of $\xi$. The preference includes the variance of $\xi$ that worsens (resp. improves) the risk situation of the criterion if $\theta>0$ (resp. $\theta<0$). Therefore, a decision maker possesses the risk-averse attitude if $\theta>0$, on the opposite, she is risk-seeking if
$\theta<0$. The risk-neutral attitude of the decision maker corresponds to
$\mathcal{E}_{\theta}[\xi]=\mathbb{E}[\xi]$ since $\mathcal{E}_{\theta}%
[\xi]\rightarrow\mathbb{E}[\xi]$ as $\theta\rightarrow0$ (see e.g.
\cite{Bertsekas}). In economics, if $\xi$ represents an intertemporal return,
Hansen et al. \cite{Hansen-Sargent-2006} pointed out that
(\ref{intro-risk-cost}) promotes robustness to model misspecification
by enhancing the decision maker's sensitivity to risk.
Therefore, a plenty of robust decision problems with the criterion (\ref{intro-risk-cost}) are closely related to the risk-sensitive control problems.
For example, Jacobson \cite{Jacobson} and Whittle \cite{Whittle} showed that the risk-sensitive control law can be computed by solving a robust penalty problem.

The risk-sensitivity has been introduced into the control problems since the
early work of Jacobson \cite{Jacobson}, Whittle \cite{Whittle}, and after that
many researchers studied this subject (see \cite{Bensoussan-Frehse-Nagai,
Fleming-McEneaney, James, LimZhou05, Moon-2019, Nagai-1996, Whittle90} and references
therein). The risk-sensitive optimal control problems aim at minimizing
\begin{equation}
J(u(\cdot)):=\mathbb{E}\left[  \exp\left\{  \theta\left(  \Phi(X(T))+\int%
_{0}^{T}f(t,X(t),u(t))dt\right)  \right\}  \right]
\label{intro-exp-objec}%
\end{equation}
over all admissible controls,
where $\theta>0$ is the risk-sensitive parameter and the state $X(\cdot)$ satisfies
the controlled stochastic differential equation (SDE):
\begin{equation}
\left\{
\begin{array}
[c]{rl}%
dX(t)= & b(t,X(t),u(t))dt+\sigma(t,X(t),u(t))dW(t),\\
\ X(0)= & x_{0}.
\end{array}
\right.  \label{intro-fward-state-eq}%
\end{equation}
The SDE (\ref{intro-fward-state-eq}) is driven by a $d$-dimensional ($d>1$)
Brownian motion $W=(W_{1}(t),$ $W_{2}(t),...,W_{d}(t))_{0\leq t\leq T}%
^{\intercal}$ with the initial data $x_{0}$; the coefficients $b$, $\sigma$,
$\Phi$, $f$ are measurable, deterministic functions in suitable dimension.

In the existing literature, it is always assumed that the decision maker possesses the same risk
attitudes toward different risk sources. In more detail, the
decision maker has the identical risk-sensitive parameter $\theta$ while
confronting different risk sources $W_{i}(t),i=1,...,d$. However, this is an
almost impossible thing to happen in reality. From the perspective of
stochastic differential utility \cite{Duffie-Epstein}, the authors in \cite{Quenez03} introduced the
asymmetry in risk aversion. To differentiate the attitudes toward risks depending
on their sources, they assume that each component of the $d$-dimensional standard
Brownian motion $W$ is an independent source of the consumption shock. For
example, if $d=2$ then $W_{1}$ may represent weather shock while $W_{2}$ may
represent health shocks, and a consumer usually shows different risk aversion
towards different risk sources. So it is an interesting issue that how to formulate the stochastic control
problems under asymmetric risk-sensitivity.

It is obvious that the risk-sensitive criterion (\ref{intro-exp-objec}) can
not characterize the asymmetric risk-sensitivity
towards different risk sources. Therefore, we must reconstruct the asymmetric
risk-sensitive criteria. Note that in \cite{Karoui-Hamadene} the risk-sensitive
control for the BSDE objective functional with quadratic growth coefficient
was considered, that is, (\ref{intro-exp-objec}) can be equivalently described
by the solution $Y(\cdot)$ of the following BSDE at time $0$:
\begin{equation}
\left\{
\begin{array}
[c]{rl}%
dY(t)= & -\left[  \frac{\theta}{2}\left\vert Z(t) \right\vert^{2}+f(t,X(t),u(t))\right]
dt+Z^{\intercal}(t)dW(t),\\
\ Y(T)= & \Phi(X(T)).
\end{array}
\right.  \label{intro-bward-state-eq-symmetric}%
\end{equation}
The risk-sensitive criterion is defined by
\begin{equation}
\tilde{J}(u(\cdot)):=Y(0),
\label{intro-obje-recursive}%
\end{equation}
where $Y(\cdot)$ is the solution of (\ref{intro-bward-state-eq-symmetric}). It is easy to verify that $Y(0)=\frac{1}{\theta}\log J(u(\cdot))$.

Notice that $\frac{\theta}{2}\left\vert Z(t) \right\vert^{2}$ can be rewritten as $Z^{\intercal}(t)\Gamma Z(t)$ with $\Gamma=\mathrm{diag}\{ \overbrace{\theta/2,\ldots,\theta/2}^{d} \}$.
The risk-sensitive parameter $\theta$ implies that the decision maker possesses the identical risk attitude to different risk sources, which cannot reflect the asymmetric risk-sensitivity arising in reality.
To characterize the asymmetric risk-sensitivity, a natural idea is to set $\Gamma=\mathrm{diag}\{ \gamma_{1}, \ldots, \gamma_{d} \}$ with different $\gamma_{i}>0$, $i=1,\ldots,d$, which is one of our main findings in this paper.
Consequently, $\frac{\theta}{2}\left\vert Z(t) \right\vert^{2}$ is replaced with $Z^{\intercal}(t)\Gamma Z(t)$ in (\ref{intro-bward-state-eq-symmetric}). 
Under this setting, it is worth pointing out that we cannot obtain a representation similar to
(\ref{intro-exp-objec}) through the exponential transformation for $Y(\cdot)$. In this sense, the asymmetric risk-sensitive criterion proposed in this paper can only be defined by
the quadratic BSDE
\begin{equation}
\left\{
\begin{array}
[c]{rl}%
dY(t)= & -\left[  Z^{\intercal}(t)\Gamma Z(t)+f(t,X(t),u(t))\right]
dt+Z^{\intercal}(t)dW(t),\\
\ Y(T)= & \Phi(X(T)),
\end{array}
\right.  \label{intro-bward-state-eq}%
\end{equation}
where $\Gamma$ is a strictly positive definite matrix.

Based on the proposed criterion, we introduce the asymmetric risk-sensitive stochastic control problems as follows: the goal
is to minimize the asymmetric risk-sensitive criterion (\ref{intro-obje-recursive}) subject to the controlled forward-backward
stochastic differential equation (FBSDE) (see \cite{Karoui-Hamadene,
Elkaroui-PQ, Peng-1993} for more details about controlled FBSDEs)
\begin{equation}
\left\{
\begin{array}
[c]{rl}%
dX(t)= & b(t,X(t),u(t))dt+\sigma(t,X(t),u(t))dW(t),\\
dY(t)= & -\left[  Z^{\intercal}(t)\Gamma Z(t)+f(t,X(t),u(t))\right]
dt+Z^{\intercal}(t)dW(t),\\
\ Y(T)= & \Phi(X(T)).
\end{array}
\right.  \label{intro-FBSDE-state-eq}%
\end{equation}
When $\Gamma=\frac{\theta}{2}\mathrm{I}_{d\times d}$, (\ref{intro-obje-recursive}) together with (\ref{intro-FBSDE-state-eq}) degenerates into the classical risk-sensitive stochastic control problems (\ref{intro-exp-objec})-(\ref{intro-fward-state-eq}). In addition, when $\Gamma=0$, it degenerates into the risk-neutral case studied by Peng \cite{Peng90}.

Now we review the solution to the classical risk-sensitive control problem (\ref{intro-exp-objec})-(\ref{intro-fward-state-eq}). Note that the cost functional (\ref{intro-exp-objec}) is not a risk-neutral form, so
the classical results of stochastic control theory cannot be directly applied.
Under the assumption that $(\Phi,f)$ is uniformly bounded, by extending state
variables and logarithmic transformation, Lim and Zhou \cite{LimZhou05} rewrote (\ref{intro-exp-objec})-(\ref{intro-fward-state-eq}) as a risk-neutral form and obtained a new global
risk-sensitive maximum principle (MP) for (\ref{intro-exp-objec})-(\ref{intro-fward-state-eq}).
The existence of a smooth solution to the associated HJB equation is proved under
restrictive regularity conditions both in the case
of control-independent diffusions \cite{Nagai-1996} and in the case of
control-dependent diffusions \cite{Fleming-Soner}.
Relying on the results of parabolic PDEs, weaker regularity of solutions is obtained by Bensoussan, Frehse, and Nagai
\cite{Bensoussan-Frehse-Nagai} in the case of control-independent diffusions.
We emphasize that Moon \cite{Moon-2021} studied the risk-sensitive control
(\ref{intro-obje-recursive}), (\ref{intro-FBSDE-state-eq}) with $\Gamma
=\frac{\theta}{2}\mathrm{I}_{d \times d}$ and control-dependent diffusions by adopting the dynamic programming principle (DPP) approach.


Moreover, for the linear-quadratic (LQ) case, Hansen and Sargent \cite{Hansen-Sargent-1995} considered a discrete-time
LQ, Gaussian risk-sensitive control problem with discounting.
A continuous-time risk-sensitive LQ problem was studied by Lim and Zhou \cite{LimZhou05}. 
Duncan \cite{Duncan} solves the LQ counterpart of (\ref{intro-exp-objec})-(\ref{intro-fward-state-eq}) by using a completion-of-square approach instead of MP or DPP.
Nevertheless, in the case of asymmetric risk-sensitive LQ
control, one cannot follow the ideas in \cite{Duncan} by carrying out the
completion-of-squares approach in the exponential in (\ref{intro-exp-objec})
because the logarithmic (equivalently, exponential) transformation fails as we mentioned early.

The main contributions in this paper are as follows: Firstly, a new criterion that describes the asymmetric risk-sensitivity is proposed, which can only be defined through the initial value of the solutions of quadratic BSDEs.
Not only can it interpret the asymmetric risk attitudes toward different risk sources emerging in economics and finance, but also provides a practical basis of the problem (\ref{intro-obje-recursive}) governed by (\ref{intro-FBSDE-state-eq}).
Due to the failure of the logarithmic transformation in the asymmetric risk-sensitive case, it is difficult to obtain the Taylor expansion for multivariate risk-sensitive parameters.
By applying the convex perturbation method, we obtain the Taylor expansion for the asymmetric risk-sensitive criterion. To further explore the meaning of $\Gamma$ in terms of risk-sensitivity
like (\ref{intro-expan}), the Taylor expansion for the asymmetric
risk-sensitive criterion is a key issue.
Without loss of generality, we put $d=2$ and $\Gamma$ be a diagonal matrix whose entries $\gamma_{1},\gamma_{2}>0$ and $\gamma_{1}\neq\gamma_{2}$ to illustrate our results. 
The Taylor expansion is ultimately expressed by
\begin{equation}
\mathcal{E}_{\gamma_{1},\gamma_{2}}[\xi]=\mathbb{E}[\xi]+\gamma_{1}%
\mathrm{D}_{1}[\xi]+\gamma_{2}\mathrm{D}_{2}[\xi]+o\left(  \sqrt{\gamma
_{1}^{2}+\gamma_{2}^{2}}\right)  ,
\label{intro-expan-asymmetric}
\end{equation}
where the functionals $(\mathrm{D}_{1},\mathrm{D}_{2})$, inheriting some
axiomatic properties that $\mathrm{Var}[\cdot]$ possesses, is called a
variance decomposition on the domain of $\mathcal{E}_{\gamma_{1},\gamma_{2}}$
such that $\mathrm{Var}[\xi]=\mathrm{D}_{1}[\xi]+\mathrm{D}_{2}[\xi]$ (see Section 2 for the
details). The result illustrates the asymmetry of risk attitudes (i.e.,
$\gamma_{i}$, $i=1,2$) for different risks with various weights $\mathrm{D}%
_{i}[\xi]$, for $i=1,2$. Particularly, (\ref{intro-expan-asymmetric})
degenerates into (\ref{intro-expan}) when $\gamma_{1}=\gamma_{2}=\frac{\theta
}{2}$. Actually, for any given strictly positive definite matrix $\Gamma$, (\ref{intro-obje-recursive}) is actually a quadratic filtration-consistent nonlinear expectation (see \cite{Hu-Ma-Peng-Yao, Ma-Yao} for more details) induced by (\ref{intro-bward-state-eq}), denoted by $\mathcal{E}_{\Gamma}$, of the random variable $\xi=\Phi(X(T))+\int_{0}^{T}g(t,X(t),u(t))dt$, i.e. $\mathcal{E}_{\Gamma}[\xi]=Y(0)$. 


Second, an asymmetric risk-sensitive LQ control problem is solved by developing a novel completion-of-squares technique for controlled quadratic BSDEs.
Applying the derived MP informally, a candidate optimal control $\bar{u}(\cdot)$ with feedback type is determined by a new Riccati differential equation.
In contrast to the classical risk-sensitive counterpart introduced by Duncan \cite{Duncan}, this Riccati equation has a nontrivial term reflecting the asymmetric risk sensitivity
and degenerates into the former when $\Gamma=\frac{\theta}{2}\mathrm{I}_{d \times d}$.
Unfortunately, the usual completion-of-squares method fails to prove the optimality of $\bar{u}(\cdot)$.
Under our framework, taking advantage of the structure of the quadratic BSDEs whose generators are convex in both the states and controls, we tackle the difficulty of verifying that a Girsanov exponential is a Radon-Nikodym derivative. The completion-of-squares technique then holds.
Not only that, the novelty of our technique enables us to discuss the admissibility of $\bar{u}(\cdot)$ over 
a wider range of admissible controls than \cite{Duncan}.


Finally, as an application of the asymmetric risk-sensitive control, 
we study a dynamic portfolio optimization problem that generalizes the results studied in \cite{Kuroda-Nagai} on a finite time horizon, which can characterize different
weights for varieties of risk sources. 
The new criterion and related control problems might have more potential for extensions of risk-sensitive dynamic portfolio optimization problems in the existing literature (e.g. \cite{Pliska1999,
Kuroda-Nagai, Nagai-2004, Nagai-Peng2002})  to the asymmetric cases.

The rest of the paper is organized as follows. In Section 2, we give some
preliminaries and introduce the asymmetric risk-sensitive criterion with a new mean-variance representation. 
Section 3 sets the stage by formulating nonlinear asymmetric risk-sensitive control problems and focusing on
the case of bounded conditions and the LQ case.
As an application of the asymmetric risk-sensitive control, section 4 discusses an asymmetric risk-sensitive dynamic portfolio optimization problem.

\section{Asymmetric risk-sensitive criterion}

As mentioned in the introduction, the risk-sensitive criterion
(\ref{intro-risk-cost}) can be regarded as a nonlinear expectation
$\mathcal{E}_{\theta}$ such that it has a Taylor expansion around $\theta= 0$,
that is,
\begin{equation}
\label{Taylor-expan-sym}
\mathcal{E}_{\theta}[\xi]=\mathbb{E}\left[
\xi\right]  + \frac{\theta}{2}\mathrm{Var}\left[  \xi\right]  + O(\theta^{2}),
\end{equation}
where $\xi= \Phi(X(T)) + \int_{0}^{T} f(t,X(t),u(t)) dt$ for any admissible
control $u(\cdot)$. A natural question follows that whether we can perform a
Taylor expansion for asymmetric counterpart of (\ref{intro-risk-cost}), if we can, to what extent such a mean-variance
representation generalizes (\ref{Taylor-expan-sym}).
Our main goal in this section is to establish the asymmetric risk-sensitive criterion and obtain its mean-variance representation.
Before the start, we shall provide some preliminaries at first.

Let $(\Omega,\mathcal{F},\mathbb{P})$ be a complete probability space on which
a standard $d$-dimensional Brownian motion $W=(W_{1}(t),W_{2}(t),...W_{d}%
(t))_{0\leq t\leq T}^{\intercal}$ is defined. Assume that $\mathbb{F=}
\{\mathcal{F}_{t}, $ $0\leq t\leq T\}$ is the $\mathbb{P}$-augmentation of the
natural filtration of $W$, where $\mathcal{F}_{0}$ contains all $\mathbb{P}%
$-null sets of $\mathcal{F}$. Denote by $\mathbb{R}^{n}$ the $n$-dimensional
real Euclidean space and $\mathbb{R}^{k\times n}$ the set of $k\times n$ real
matrices. Let $\langle\cdot,\cdot\rangle$ (resp. $\left\vert \cdot\right\vert
$) denote the usual scalar product (resp. usual norm) of $\mathbb{R}^{n}$ and
$\mathbb{R}^{k\times n}$. The scalar product (resp. norm) of $A=(a_{ij})$,
$B=(b_{ij})\in\mathbb{R}^{k\times n}$ is denoted by $\langle A,B\rangle
=\mathrm{tr}\{AB^{\intercal}\}$ (resp. $\vert A\vert=\sqrt{\mathrm{tr}%
\{AA^{\intercal}\}}$), where the superscript $^{\intercal}$ denotes the
transpose of vectors or matrices. Denote by $\mathbb{S}^{n\times n}$ the set
of all $n\times n$ real symmetric matrices and $\mathrm{I}_{n \times n}$ the
$n \times n$ identity matrix.

For any given real number $p\geq1$ and positive integer $m$, we introduce the following spaces.
 $L^{\infty}([0,T];\mathbb{R}^{n})$: the space of
	$\mathbb{R}^{n}$-valued measurable functions $f(\cdot)$ on $[0,T]$ such
	that
	\[
	||f(\cdot)||_{\infty}:=\mathrm{sup}_{t \in\lbrack0,T] }|f(t)|<+\infty.
	\]
 $C([0,T];\mathbb{R}^{n})$: the space of $\mathbb{R}^{n}$-valued continuous functions $f(\cdot)$ on $[0,T]$;
$C^{m}([0,T];$ $\mathbb{R}^{n})$: the space of $\mathbb{R}^{n}$-valued functions $f(\cdot)$ on $[0,T]$
	that is $m$-times continuously differentiable, and the $m$th-order derivative is denoted by $f^{(m)}(\cdot)$; $L^{p}(\mathcal{F}_{T};\mathbb{R}^{n})$ : the space of $\mathcal{F}%
	_{T}$-measurable $\mathbb{R}^{n}$-valued random vectors $\eta$ such that
	\[
	\mathbb{E}[|\eta|^{p}]<+\infty.
	\]
$L^{\infty}(\mathcal{F}_{T};\mathbb{R}^{n})$: the space of
	$\mathcal{F}_{T}$-measurable $\mathbb{R}^{n}$-valued random vectors $\eta$
	such that, $\mathbb{P}$-a.s.
	\[
	\mathrm{ess~sup}_{\omega\in\Omega}|\eta(\omega)|<+\infty.
	\]
$L_{\mathbb{F}}^{\infty}([0,T];\mathbb{R}^{n})$: the space of
	$\mathbb{F}$-adapted $\mathbb{R}^{n}$-valued stochastic processes $f(\cdot)$ on $[0,T]$
	such that, $\lambda \otimes \mathbb{P}$-a.e.
	\[
	\mathrm{ess~sup}_{(t,\omega)\in\lbrack0,T]\times\Omega
	}|f(t,\omega)|<+\infty;
	\]
	where $\lambda$ represents the Lebesgue measure on $[0,T]$.
 $L_{\mathbb{F}}^{p,q}([0,T];\mathbb{R}^{n})$: the space of
	$\mathbb{F}$-adapted $\mathbb{R}^{n}$-valued stochastic processes $f(\cdot)$ on $[0,T]$
	such that
	\[
	\mathbb{E}\left[  \left(  \int_{0}^{T}%
	|f(t)|^{p}dt\right)  ^{\frac{q}{p}}\right] <+\infty;
	\]
	and when $p=q$, we simply write $L_{\mathbb{F}}^{p}([0,T];\mathbb{R}^{n})$
	rather than $L_{\mathbb{F}}^{p,q}([0,T];\mathbb{R}^{n})$.
 $L_{\mathbb{F}}^{p}(\Omega;$ $C([0,T],\mathbb{R}^{n}))$: the space of
	$\mathbb{F}$-adapted, $\mathbb{R}^{n}$-valued continuous stochastic processes $f(\cdot)$
	on $[0,T]$ such that
	\[
	\mathbb{E}\left[  \sup_{t \in [0,T]}|f(t)|^{p}\right]  <+\infty.
	\]


To describe the asymmetric risk-sensitivity, without loss of generality,
we consider the diagonal matrix $\Gamma= \mathrm{diag}\{ \gamma_{1}, \ldots,
\gamma_{d} \}$ for some $(\gamma_{1}, \ldots, \gamma_{d}) \in\mathbb{R}^{d}$
such that $\gamma_{i} >0, i=1,\ldots,d$. For simplicity, we put $d=2$ and
the analysis in the case $d>2$ is similar to the former case.
Consider the quadratic BSDE
\begin{equation}
	\left\{
	\begin{array}
	[c]{rl}%
	dY^{\gamma_{1},\gamma_{2}}(t)= & -\left[  \gamma_{1}\left\vert Z_{1}^{\gamma_{1},\gamma_{2}
	}(t)\right\vert ^{2}+\gamma_{2}\left\vert Z_{2}^{\gamma_{1},\gamma_{2}}(t)\right\vert
	^{2}\right]  dt+Z_{1}^{\gamma_{1},\gamma_{2}}(t)dW_{1}(t)\\
&+Z_{2}^{\gamma_{1},\gamma_{2}}%
	(t)dW_{2}(t),\\
	\ Y^{\gamma_{1},\gamma_{2}}(T)= & \xi.
	\end{array}
	\right.  \label{asymmetric-criterion-eq}%
\end{equation}

\begin{assumption}
\label{Taylor-terminal-condition} $\xi$ is an $\mathcal{F}_{T}$-measurable
random variable such that $\mathbb{E} \left[  e^{16 \left\vert \xi\right\vert
} \right]  < +\infty$.
\end{assumption}

The following lemma is an application of Corollary 4 in \cite{Briand-H} and
Theorem 3.3 in \cite{Delbaen-Hu-Richou2011} to (\ref{asymmetric-criterion-eq}).

\begin{lemma}
\label{lem-Taylor-state-uniform-bound} Let Assumption
\ref{Taylor-terminal-condition} hold. If $(\gamma_{1}, \gamma_{2})\in
\lbrack0,1]\times\lbrack0,1]$, then the state equation (\ref{asymmetric-criterion-eq})
admits a unique solution $\left(  Y^{\gamma_{1},\gamma_{2}}(\cdot),Z^{\gamma_{1},\gamma_{2}}%
(\cdot) \right)  $ such that
\[ \mathbb{E}\left[  e^{16\sup_{t\in\lbrack
0,T]}\left\vert Y^{\gamma_{1},\gamma_{2}}(t)\right\vert } \right]  <+\infty
\]
 and
$Z^{\gamma_{1},\gamma_{2}}(\cdot)= \left(  Z_{1}^{\gamma_{1},\gamma_{2}}(\cdot), Z_{2}%
^{\gamma_{1},\gamma_{2}}(\cdot) \right)  \in L_{\mathbb{F}}^{2}([0,T];\mathbb{R}^{2})$.
Moreover, there exists a $C>0$ such that
\begin{equation}
\mathbb{E}\left[  \exp\left\{  16\sup_{t\in\lbrack0,T]}\left\vert
Y^{\gamma_{1},\gamma_{2}}(t)\right\vert \right\}  +\left(  \int_{0}^{T}\left\vert
Z^{\gamma_{1},\gamma_{2}}(t)\right\vert ^{2}dt\right)  ^{4}\right]  \leq C\mathbb{E}%
\left[  e^{16\left\vert \xi\right\vert }\right]  , \label{Taylor-state-est}%
\end{equation}
where $C$ depends only on $T$.
\end{lemma}

For any given $(\gamma_{1}, \gamma_{2})\in\lbrack0,1]\times\lbrack
0,1]$, Lemma \ref{lem-Taylor-state-uniform-bound} guarantees the
well-posedness of quadratic BSDE (\ref{asymmetric-criterion-eq}). 
According to Definition 3.3 and Example 3.4 in \cite{Hu-Ma-Peng-Yao}, the unique solution
$Y^{\gamma_{1}, \gamma_{2}}(\cdot)$ to (\ref{asymmetric-criterion-eq}) actually induces a quadratic
$\mathbb{F}$-consistent nonlinear expectation $\mathcal{E}_{\gamma_{1},\gamma_{2}}$ with
its domain $\mathrm{Dom}(\mathcal{E}_{\gamma_{1},\gamma_{2}})$ such that
\begin{equation}
\label{quadratic-g-expectation}%
\begin{array}
[c]{rl}%
\mathcal{E}_{\gamma_{1},\gamma_{2}}[\xi]:= & Y^{\gamma_{1},\gamma_{2}}(0),\text{ }\forall \xi\in
\mathrm{Dom}(\mathcal{E}_{\gamma_{1},\gamma_{2}}),\\
\mathrm{Dom}(\mathcal{E}_{\gamma_{1},\gamma_{2}}):= & \left\{  \xi\in L^{2}%
(\mathcal{F}_{T};\mathbb{R}),\text{ }\mathbb{E}[e^{16\left\vert \xi\right\vert
}]<+\infty\right\}  .
\end{array}
\end{equation}


\begin{definition}
	\label{def-asymmetric-criterion}
For any random variable $\xi\in \mathrm{Dom}(\mathcal{E}_{\gamma_{1},\gamma_{2}})$, the nonlinear expectation $\mathcal{E}_{\gamma_{1},\gamma_{2}}[\xi]$
is called the asymmetric risk-sensitive criterion with respect to $\xi$.
\end{definition}

We introduce an auxiliary control problem and resort to a variational method, which is usually used in deriving the stochastic MP,
to perform the Taylor expansion for $\mathcal{E}_{\gamma_{1},\gamma_{2}}[\xi]$.
In the following context, the constant $C$ may change from line to line in the
proofs.

Consider the controlled BSDE
\begin{equation*}
\left\{
\begin{array}
[c]{rl}%
dY^{v_{1},v_{2}}(t)= & -\left[  v_{1}\left\vert Z_{1}^{v_{1},v_{2}%
}(t)\right\vert ^{2}+v_{2}\left\vert Z_{2}^{v_{1},v_{2}}(t)\right\vert
^{2}\right]  dt+Z_{1}^{v_{1},v_{2}}(t)dW_{1}(t)\\
&+Z_{2}^{v_{1},v_{2}}%
(t)dW_{2}(t),\\
\ Y^{v_{1},v_{2}}(T)= & \xi,
\end{array}
\right.
\end{equation*}
where control variables $(v_{1}, v_{2})\in\lbrack0,1]\times\lbrack0,1]$.
The objective is to
minimize
\begin{equation}
J(v_{1}, v_{2}):=Y^{v_{1},v_{2}}(0) \label{Taylor-obje}%
\end{equation}
over $(v_{1}, v_{2})\in\lbrack0,1]\times\lbrack0,1]$.
One can observe that the couple $\left(  \bar{v}_{1},
\bar{v}_{2} \right)  = \left(  0,0 \right)  $ minimize (\ref{Taylor-obje})
uniquely, and the corresponding optimal trajectory, denoted by $\left(
\bar{Y}(\cdot), \bar{Z}_{1}(\cdot),\right.$ $ \left. \bar{Z}_{2}(\cdot) \right)$, satisfies
the following BSDE:
\begin{equation}
\left\{
\begin{array}
[c]{rl}%
d\bar{Y}(t)= & \bar{Z}_{1}(t)dW_{1}(t)+\bar{Z}_{2}(t)dW_{2}(t),\\
\ \bar{Y}(T)= & \xi.
\end{array}
\right.
\label{Taylor-opti-eq}%
\end{equation}
Since $[0,1]\times[0,1] $ is closed and convex, we adopt the convex
perturbation around $\left(  0,0 \right)  $ to deduce the variational equation for
this auxiliary control problem. For any $\gamma_{1}, \gamma_{2} \in[0,1]$, set
$v^{\gamma_{i}}=\bar{v}_{i}+\gamma_{i}(1-\bar{v}_{i}),i=1,2$. It is obvious that
$\left(  v^{\gamma_{1}}, v^{\gamma_{2}} \right)  = \left(  \gamma_{1},
\gamma_{2} \right) $.
Combing (\ref{Taylor-opti-eq}), we have
\begin{equation}%
\begin{array}
[c]{rl}%
Y^{\gamma_{1},\gamma_{2}}(t)-\bar{Y}(t)= & \displaystyle \int_{t}^{T}\left(
\gamma_{1}\left\vert Z_{1}^{\gamma_{1},\gamma_{2}}(s)\right\vert ^{2}%
+\gamma_{2}\left\vert Z_{2}^{\gamma_{1},\gamma_{2}}(s)\right\vert ^{2}\right)
ds\\
& -\displaystyle \int_{t}^{T}\left(  Z_{1}^{\gamma_{1},\gamma_{2}}(s)-\bar
{Z}_{1}(s)\right)  dW_{1}(s)\\
&- \displaystyle \int_{t}^{T}\left(  Z_{2}%
^{\gamma_{1},\gamma_{2}}(s)-\bar{Z}_{2}(s)\right)  dW_{2}(s),
\end{array}
\label{Taylor-expan-1}%
\end{equation}
We provide the estimate for (\ref{Taylor-expan-1}) by the following lemma.

\begin{lemma}
\label{lem-Taylor-expan1-est} Let Assumption \ref{Taylor-terminal-condition}
hold. Then
\[
\mathbb{E}\left[  \sup_{t\in\lbrack0,T]}\left\vert Y^{\gamma_{1},\gamma_{2}%
}(t)-\bar{Y}(t)\right\vert ^{4}+\left(  \int_{0}^{T}\left\vert Z_{{}}%
^{\gamma_{1},\gamma_{2}}(t)-\bar{Z}(t)\right\vert ^{2}dt\right)  ^{2}\right]
=O\left(  \left(  \gamma_{1}^{2}+\gamma_{2}^{2}\right)  ^{2}\right)  ,
\]
where $Z^{\gamma_{1},\gamma_{2}}(\cdot)-\bar{Z}(\cdot)=\left(  Z_{1}%
^{\gamma_{1},\gamma_{2}}(\cdot)-\bar{Z}_{1}(\cdot),Z_{2}^{\gamma_{1}%
,\gamma_{2}}(\cdot)-\bar{Z}_{2}(\cdot)\right)  $.
\end{lemma}

\begin{proof}
Applying a standard BSDE estimate (please refer to \cite{Elkaroui-PQ, Yong-Zhou}) to (\ref{Taylor-expan-1}),
we have
\[%
\begin{array}
[c]{l}%
\mathbb{E}\left[  \sup\limits_{t\in\lbrack0,T]}\left\vert Y^{\gamma_{1}%
,\gamma_{2}}(t)-\bar{Y}(t)\right\vert ^{4}+\left(  \int_{0}^{T}\left\vert
Z_{{}}^{\gamma_{1},\gamma_{2}}(t)-\bar{Z}(t)\right\vert ^{2}dt\right)
^{2}\right]  \\
\leq C\mathbb{E}\left[  \left(  \int_{0}^{T}\left(  \gamma_{1}\left\vert
Z_{1}^{\gamma_{1},\gamma_{2}}(t)\right\vert ^{2}+\gamma_{2}\left\vert
Z_{2}^{\gamma_{1},\gamma_{2}}(t)\right\vert ^{2}\right)  dt\right)
^{4}\right]  \\
\leq C\mathbb{E}\left[  \left(  \int_{0}^{T}\left\vert Z_{{}}^{\gamma
_{1},\gamma_{2}}(t)\right\vert ^{2}dt\right)  ^{4}\right]  \left(  \gamma
_{1}^{2}+\gamma_{2}^{2}\right)  ^{2},
\end{array}
\]
where $Z_{{}}^{\gamma_{1},\gamma_{2}}(\cdot)= \left( Z_{1}^{\gamma_{1},\gamma_{2}}(\cdot), Z_{2}^{\gamma_{1},\gamma_{2}}(\cdot) \right)$,
and $C$ depends  only on $T$.
By Lemma \ref{lem-Taylor-state-uniform-bound}, we conclude that
$$\sup_{\gamma_{1},\gamma_{2}\in\lbrack0,1]}\mathbb{E}\left[  \left(  \int%
_{0}^{T}\left\vert Z_{{}}^{\gamma_{1},\gamma_{2}}(t)\right\vert ^{2}dt\right)
^{4}\right]  \leq C\mathbb{E}\left[  e^{16\left\vert \xi\right\vert }\right],$$
where $C$ depends only on $T$. Therefore, we finally obtain
\[
\mathbb{E}\left[  \sup\limits_{t\in\lbrack0,T]}\left\vert Y^{\gamma_{1}%
,\gamma_{2}}(t)-\bar{Y}(t)\right\vert ^{4}+\left(  \int_{0}^{T}\left\vert
Z_{{}}^{\gamma_{1},\gamma_{2}}(t)-\bar{Z}(t)\right\vert ^{2}dt\right)
^{2}\right]  \leq C\left(  \gamma_{1}^{2}+\gamma_{2}^{2}\right)  ^{2},
\]
where $C$ depends only on $T$ and $\xi$.
\end{proof}

For $i=1,2$, let $\left(  Y_{i}(\cdot), Z_{i1}(\cdot), Z_{i2}(\cdot) \right)
$ be respectively the solution to the following BSDEs:
\begin{equation}
\left\{
\begin{array}
[c]{rl}%
dY_{i}(t)= & -\left\vert \bar{Z}_{i}(t)\right\vert ^{2}dt+Z_{i1}%
(t)dW_{1}(t)+Z_{i2}(t)dW_{2}(t),\\
Y_{i}(T)= & 0.
\end{array}
\right.  \label{Taylor-expan-2}%
\end{equation}
Under Assumption \ref{Taylor-terminal-condition}, the well-posedness of
(\ref{Taylor-expan-2}) can be guaranteed by the classical theory of the BSDEs
(please refer to \cite{Elkaroui-PQ, Yong-Zhou}) and the estimate
(\ref{Taylor-state-est}) holds.

Now we can state the main result of this section.

\begin{theorem}
\label{thm-Taylor-expan} Let Assumption \ref{Taylor-terminal-condition} hold.
Then
\[%
\begin{array}
[c]{l}%
\mathbb{E}\left[  \sup\limits_{t\in\lbrack0,T]}\left\vert Y^{\gamma_{1}%
,\gamma_{2}}(t)-\bar{Y}(t)-\gamma_{1}Y_{1}(t)-\gamma_{2}Y_{2}(t)\right\vert
^{2}\right. \\
+\left.  \sum\limits_{i=1}^{2} \displaystyle\int_{0}^{T}  \left\vert
Z_{i}^{\gamma_{1},\gamma_{2}}(t)-\bar{Z}_{i}(t)-\gamma_{1}Z_{1i}(t)-\gamma
_{2}Z_{2i}(t)\right\vert ^{2}  dt\right]  =o\left(  \gamma_{1}%
^{2}+\gamma_{2}^{2}\right)  .
\end{array}
\]

\end{theorem}

\begin{proof}
Denote $\eta(\cdot)=Y^{\gamma_{1},\gamma_{2}}(\cdot)-\bar{Y}(\cdot)-\gamma
_{1}Y_{1}(\cdot)-\gamma_{2}Y_{2}(\cdot)$ and $\zeta_{i}(\cdot)=$
$Z_{i}^{\gamma_{1},\gamma_{2}}(\cdot)-\bar{Z}_{i}(\cdot)-\gamma_{1}%
Z_{1i}(\cdot)-\gamma_{2}Z_{2i}(\cdot)$ for $i=1,2$. From (\ref{Taylor-expan-1}%
) and (\ref{Taylor-expan-2}), we get
\begin{equation}%
\begin{array}
[c]{rl}%
\eta(t)= & \sum\limits_{i=1}^{2} \displaystyle\int_{t}^{T}\left(  \gamma
_{i}(Z_{i}^{\gamma_{1},\gamma_{2}}(s)+\bar{Z}_{i}(s))(Z_{i}^{\gamma_{1}%
,\gamma_{2}}(s)-\bar{Z}_{i}(s))\right)  ds\\
& -\displaystyle\int_{t}^{T}\zeta_{1}(s)dW_{1}(s)-\int_{t}^{T}\zeta
_{2}(s)dW_{2}(s).
\end{array}
\label{Taylor-expan-3}%
\end{equation}
Similar to the estimate (\ref{Taylor-expan-1}), by using a standard BSDE
estimate, we have
\bigskip%
\[%
\begin{array}
[c]{l}%
\mathbb{E}\left[  \sup\limits_{t\in\lbrack0,T]}\left\vert \eta(t)\right\vert
^{2}+\displaystyle \int_{0}^{T}\left(  \left\vert \zeta_{1}(t)\right\vert ^{2}+\left\vert
\zeta_{2}(t)\right\vert ^{2}\right)  dt\right]  \\
\leq C\mathbb{E}\left[  \left( \sum\limits_{i=1}^{2} \displaystyle\int_{0}^{T}\gamma_{i}\left\vert Z_{i}^{\gamma_{1},\gamma_{2}}(t)+\bar{Z}%
_{i}(t)\right\vert \left\vert Z_{i}^{\gamma_{1},\gamma_{2}}(t)-\bar{Z}%
_{i}(t)\right\vert dt\right)  ^{2}\right]  \\
\leq C\sum\limits_{i=1}^{2}\mathbb{E}\left[  \left(  \displaystyle\int_{0}%
^{T}\left\vert Z_{i}^{\gamma_{1},\gamma_{2}}(t)+\bar{Z}_{i}(t)\right\vert
\left\vert Z_{i}^{\gamma_{1},\gamma_{2}}(t)-\bar{Z}_{i}(t)\right\vert
dt\right)  ^{2}\right]  \left(  \gamma_{1}^{2}+\gamma_{2}^{2}\right)  .
\end{array}
\]
For $i=1,2$, by H\"{o}lder's inequality, the estimate (\ref{Taylor-state-est})
and Lemma \ref{lem-Taylor-expan1-est}, we have
\[%
\begin{array}
[c]{l}%
\mathbb{E}\left[  \left(  \displaystyle\int_{0}^{T}\left\vert Z_{i}%
^{\gamma_{1},\gamma_{2}}(t)+\bar{Z}_{i}(t)\right\vert \left\vert Z_{i}%
^{\gamma_{1},\gamma_{2}}(t)-\bar{Z}_{i}(t)\right\vert dt\right)  ^{2}\right]
\\
\leq\mathbb{E}\left[  \left(  \displaystyle\int_{0}^{T}\left\vert
Z_{i}^{\gamma_{1},\gamma_{2}}(t)+\bar{Z}_{i}(t)\right\vert ^{2}dt\right)
\left( \displaystyle \int_{0}^{T}\left\vert Z_{i}^{\gamma_{1},\gamma_{2}}(t)-\bar{Z}%
_{i}(t)\right\vert ^{2}dt\right)  \right]  \\
\leq2\sqrt{2}\left(  \mathbb{E}\left[  \left(  \displaystyle\int_{0}%
^{T}\left\vert Z_{i}^{\gamma_{1},\gamma_{2}}(t)\right\vert ^{2}dt\right)
^{2}\right]  +\mathbb{E}\left[  \left(  \displaystyle\int_{0}^{T}\left\vert
\bar{Z}_{i}(t)\right\vert ^{2}dt\right)  ^{2}\right]  \right)  ^{\frac{1}{2}%
}\\
\quad \ \ \cdot \left(  \mathbb{E}\left[  \left(  \displaystyle\int_{0}^{T}\left\vert
Z_{i}^{\gamma_{1},\gamma_{2}}(t)-\bar{Z}_{i}(t)\right\vert ^{2}dt\right)
^{2}\right]  \right)  ^{\frac{1}{2}}\\
\leq C\left(  \gamma_{1}^{2}+\gamma_{2}^{2}\right)  .
\end{array}
\]
Hence, we finally obtain
\[
\mathbb{E}\left[  \sup\limits_{t\in\lbrack0,T]}\left\vert \eta(t)\right\vert
^{2}+\displaystyle\int_{0}^{T}\left(  \left\vert \zeta_{1}(t)\right\vert
^{2}+\left\vert \zeta_{2}(t)\right\vert ^{2}\right)  dt\right]  \leq C\left(
\gamma_{1}^{2}+\gamma_{2}^{2}\right)  ^{2},
\]
where $C$ depends only on $T$ and $\xi$.
The proof is complete.
\end{proof}

Thanks to Theorem \ref{thm-Taylor-expan}, the following Taylor expansion for
$Y^{\gamma_{1},\gamma_{2}}(0)$ holds:
\begin{equation}
Y^{\gamma_{1},\gamma_{2}}(0)= \bar{Y}(0)+\gamma_{1}Y_{1}(0)+\gamma_{2}%
Y_{2}(0)+o\left(  \sqrt{\gamma_{1}^{2}+\gamma_{2}^{2}}\right)  ,
\label{Taylor-expan-express}%
\end{equation}
where $\bar{Y}(\cdot)$, $Y_{i}(\cdot),i=1,2$ are the unique solutions to
(\ref{Taylor-opti-eq}), (\ref{Taylor-expan-2}) respectively.

\begin{remark}
A higher order Taylor expansion (e.g., second order) of $Y^{\gamma_{1},\gamma_{2}}(0)$ can also be
obtained if stronger integrability is imposed on the exponential of the terminal value
$\xi$. For example, suppose $\mathbb{E}\left[  e^{32\left\vert
\xi\right\vert }\right]  <+\infty$, then, similar to the proof of
(\ref{Taylor-expan-express}), we have
\begin{equation}%
\begin{array}
[c]{rl}%
Y^{\gamma_{1},\gamma_{2}}(0)= & \bar{Y}(0)+\gamma_{1}Y_{1}(0)+\gamma_{2}%
Y_{2}(0)\\
& +\dfrac{1}{2}\left(  \gamma_{1},\gamma_{2}\right)  \left(
\begin{array}
[c]{cc}%
Y_{11}(0) & Y_{12}(0)\\
Y_{21}(0) & Y_{22}(0)
\end{array}
\right)  \left(
\begin{array}
[c]{c}%
\gamma_{1}\\
\gamma_{2}%
\end{array}
\right)  +o\left(  \gamma_{1}^{2}+\gamma_{2}^{2}\right)  ,
\end{array}
\label{Taylor-expan-higher-order}%
\end{equation}
where $\bar{Y}(\cdot)$, $Y_{i}(\cdot) $, $i=1,2$ satisfy (\ref{Taylor-opti-eq}), (\ref{Taylor-expan-2}) respectively, and $Y_{ij}(\cdot)  $, $i,j=1,2$ satisfy the following BSDEs:
\[
\left\{
\begin{array}
[c]{rl}%
dY_{ij}(t)= & -\left[  \bar{Z}_{i}(t)Z_{ji}(t)+\bar{Z}_{j}(t)Z_{ij}(t)\right]
dt+Z_{ij1}(t)dW_{1}(t)+Z_{ij2}(t)dW_{2}(t),\\
Y_{ij}(T)= & 0,\text{ \ }i,j=1,2.
\end{array}
\right.
\]
Here $\bar{Z}(\cdot)=(\bar{Z}_{1}(\cdot),\bar{Z}_{2}(\cdot))$ satisfies (\ref{Taylor-expan-2}).
Furthermore, if $\xi$ has exponential moment of all order, then any order
Taylor expansion of $Y^{\gamma_{1},\gamma_{2}}(0)$ like (\ref{Taylor-expan-express}) and
(\ref{Taylor-expan-higher-order}) can be obtained. We omit the proof for the
lack of space.
\end{remark}

Recalling $\mathcal{E}_{\gamma_{1},\gamma_{2}}[\xi] = Y^{\gamma_{1},\gamma_{2}}(0)$,
the remaining mission is to interpret the right-hand side of (\ref{Taylor-expan-express}) from a viewpoint of mean-variance representation.
To this end, we shall introduce a characterization for a variance decomposition on the space of all square integrable random variables as follows.

Let $m,n$ be two positive integers. $\{ \mathcal{F}^{i}, i=1,\ldots,m \}$ are
sub-$\sigma$ fields of $\mathcal{F}$ and they are independent of each other.
Set $\mathcal{G}=\bigvee_{i=1}^{m}\mathcal{F}^{i}$ and it is well known that
$L^{2}(\mathcal{G};\mathbb{R}^{n})=\mathbb{R}^{n} \oplus L_{0}^{2}%
(\mathcal{G};\mathbb{R}^{n})$, $L^{2}(\mathcal{F}^{i};\mathbb{R}%
^{n})=\mathbb{R}^{n} \oplus L_{0}^{2}(\mathcal{F}^{i};\mathbb{R}^{n})$, where
\[%
\begin{array}
[c]{l}%
L_{0}^{2}(\mathcal{G};\mathbb{R}^{n}):=\left\{  \xi\in L^{2}(\mathcal{G}%
;\mathbb{R}^{n}):\mathbb{E}[\xi]=0\right\}  ,\\
L_{0}^{2}(\mathcal{F}^{i};\mathbb{R}^{n}):=\left\{  \xi\in L^{2}%
(\mathcal{F}^{i};\mathbb{R}^{n}):\mathbb{E}[\xi]=0\right\}  ,i=1,\ldots,m.
\end{array}
\]
We simply write $L^{2}(\mathcal{G};\mathbb{R}^{n})$, $L^{2}(\mathcal{F}%
^{i};\mathbb{R}^{n})$, $L_{0}^{2}(\mathcal{G};\mathbb{R}^{n})$, $L_{0}%
^{2}(\mathcal{F}^{i};\mathbb{R}^{n})$ with $L^{2}(\mathcal{G})$,
$L^{2}(\mathcal{F}^{i})$, $L_{0}^{2}(\mathcal{G})$, $L_{0}^{2}(\mathcal{F}%
^{i})$ respectively unless the dimension of the space needs to be indicated.
In addition, for any closed linear subspace $\mathcal{L} \subset
L^{2}(\mathcal{G})$, we denote by $P_{\mathcal{L}}$ the projection operator
from $L^{2}(\mathcal{G})$ upon $\mathcal{L}$ and by $\mathcal{L}^{\perp}$ the
orthogonal complement of $\mathcal{L}$ with respect to $L^{2}(\mathcal{G})$.

\begin{definition}
\label{def-var-decomp} A set of functionals $\{\mathrm{D}%
_{i},i=1,\ldots,m\}$ is called a variance decomposition on $L^{2}(\mathcal{G})$, if for
any $\xi\in L^{2}(\mathcal{G})$
\begin{equation}
\label{eq-var-decomp}\mathrm{Var}[\xi] = \sum_{i=1}^{m}\mathrm{D}_{i}\left[
\xi\right] ,
\end{equation}
where $\mathrm{D}_{i}:L^{2}(\mathcal{G}) \longmapsto\mathbb{R}$ satisfies
following axiomatic assumptions:

\begin{enumerate}

\item[(A1)] $\mathrm{D}_{i}[a \xi+ c]=a^{2}\mathrm{D}_{i}[\xi], \forall a
\in\mathbb{R}, c \in\mathbb{R}^{n}$;

\item[(A2)] $\forall\{\xi_{k}\}_{k \in\mathbb{N}_{+}} \subset L_{0}%
^{2}(\mathcal{G}), \xi\in L_{0}^{2}(\mathcal{G})$, if $\lim_{k\rightarrow
\infty}\mathbb{E}\left[  \left\vert \xi_{k}-\xi\right\vert ^{2} \right] =0$
then $$\lim_{k\rightarrow\infty}\mathrm{D}_{i}[\xi_{k}]=\mathrm{D}_{i}[\xi];$$

\item[(A3)] $\mathrm{D}_{i}[\xi] = \mathrm{Var}[\xi],\forall\xi\in L_{0}%
^{2}(\mathcal{F}^{i})$;

\item[(A4)] $\mathrm{D}_{i}[\xi] = 0, \forall\xi\in L_{0}^{2}(\mathcal{F}%
^{j}),j \neq i$;

\item[(A5)] there exists a closed linear subspace $\mathcal{L}_{i} \supset
L_{0}^{2}(\mathcal{F}^{i})$ and $\mathcal{L}_{i}^{\perp} \supset\bigcup_{j
\neq i} L_{0}^{2}(\mathcal{F}^{j})$ such that
\[
\forall\xi, \eta\in L_{0}^{2}(\mathcal{G}), \ \ \text{if}\ \ \mathrm{Cov}
\left[  P_{\mathcal{L}_{i}}(\xi), P_{\mathcal{L}_{i}}(\eta) \right]  = 0
\ \ \text{then}\ \ \mathrm{D}_{i}[\xi+ \eta] = \mathrm{D}_{i}[\xi] +
\mathrm{D}_{i}[\eta].
\]

\end{enumerate}
\end{definition}

\begin{remark}
From Definition \ref{def-var-decomp}, if $\{\mathrm{D}_{i},i=1,\ldots,m\}$ is
a variance decomposition on $L^{2}(\mathcal{G})$ then it is easy to check that
axiomatic assumptions (A2), (A3), (A4) also hold when $L_{0}^{2}(\mathcal{G}%
)$, $L_{0}^{2}(\mathcal{F}^{i})$, $L_{0}^{2}(\mathcal{F}^{j})$ are replaced by
$L^{2}(\mathcal{G})$, $L^{2}(\mathcal{F}^{i})$, $L(\mathcal{F}^{j})$
respectively, and (A5) holds for any $\xi, \eta\in L^{2}(\mathcal{G})$. In
particular, it follows from (A1), (A4), (A5) that
\[
\mathrm{D}_{i}[\xi+ \eta] = \mathrm{D}_{i}[\xi],\ \ \forall\xi\in
L^{2}(\mathcal{G}),\eta\in L^{2}(\mathcal{F}^{j}),j \neq i.
\]

\end{remark}

The following proposition provide an approach to constructing a variance decomposition on $L^{2}(\mathcal{G})$.

\begin{proposition}
\label{prop-suffi-decompose} Suppose that $L_{0}^{2}(\mathcal{G})$ admits an
orthogonal direct sum:
\[
L_{0}^{2}(\mathcal{G})=\mathcal{L}_{1} \oplus\cdots\oplus\mathcal{L}_{m}
\]
with $m$ closed linear subspaces such that $\mathcal{L}_{1} \supset L_{0}%
^{2}(\mathcal{F}^{1}), \ldots, \mathcal{L}_{m} \supset L_{0}^{2}%
(\mathcal{F}^{m})$. For $i=1,\ldots,m$, define a functional $\mathrm{D}%
_{i}:L^{2}(\mathcal{G}) \longmapsto\mathbb{R}$ such that for any $\xi\in
L^{2}(\mathcal{G})$ with the orthogonal decomposition $\xi=\mathbb{E}%
[\xi]+\sum_{i=1}^{m}\xi_{i}$,
\begin{equation}
\label{suffi-construct-var}\mathrm{D}_{i}\left[  \xi\right] :=\mathrm{Var}%
\left[  \xi_{i} \right] , \ \xi_{i} \in\mathcal{L}_{i},
\end{equation}
then $\{\mathrm{D}_{i},i=1,\ldots,m\}$ is a variance decomposition on
$L^{2}(\mathcal{G})$.
\end{proposition}

\begin{proof}
$\forall \xi \in L^{2}(\mathcal{G})$, (\ref{eq-var-decomp}) is obvious due to (\ref{suffi-construct-var}). It is not difficult to verify (A1)-(A4) so we only check (A5). For any $\xi, \eta \in L_{0}^{2}(\mathcal{G})$ with
\[
\xi=\sum_{i=1}^{m}\xi_{i},\quad \eta=\sum_{i=1}^{m}\eta_{i}, \quad \xi_{i}, \eta_{i} \in \mathcal{L}_{i}.
\]
Since
$
\mathrm{Cov} \left[ \xi_{i}, \eta_{i} \right] = \mathrm{Cov} \left[ P_{\mathcal{L}_{i}}(\xi), P_{\mathcal{L}_{i}}(\eta) \right] =  0
$,
then it follows from (\ref{suffi-construct-var}) immediately that
\begin{equation}
\label{cov-proj-zero}
\mathrm{D}_{i}[\xi + \eta] = \mathrm{Var}[\xi_{i} + \eta_{i}] = \mathrm{Var}[\xi_{i}] + \mathrm{Var}[\eta_{i}] = \mathrm{D}_{i}[\xi] + \mathrm{D}_{i}[\eta],
\end{equation}
which completes the proof.
\end{proof}

Thanks to Proposition \ref{prop-suffi-decompose}, the following example is
helpful for us to interpret how a decision maker measures the risks stemming
from different risk sources and weights each of them with her asymmetric
risk-sensitive parameters.

\begin{example}
\label{example-2} Put $m=d$. For $i=1,\ldots,d$, let $\mathbb{F}%
^{i}=\{\mathcal{F}_{t}^{i}: 0 \leq t \leq T\}$ where $\mathcal{F}_{t}%
^{i}:=\sigma(W_{i}(s): 0 \leq s \leq t)$. Denote by $L_{\mathbb{F}^{i}}%
^{2}([0,T];\mathbb{R}^{n})$ the subspace of $L_{\mathbb{F}}^{2}%
([0,T];\mathbb{R}^{n})$ such that any process in that is $\mathbb{F}^{i}$-adapted.
For $i=1,\ldots,d$, take $\mathcal{F}^{i}=\mathcal{F}_{T}^{i}$. It is obvious
that $\mathcal{F}^{i}, i=1,\ldots,d$ are mutually independent and
$\mathcal{G}:=\bigvee_{i=1}^{d}\mathcal{F}^{i}=\mathcal{F}_{T}$. On the one
hand, for any $\xi\in L_{0}^{2}(\mathcal{F}_{T})$, according to the martingale
representation theorem, there exist $\varphi_{i}(\cdot)\in L_{\mathbb{F}}%
^{2}([0,T];\mathbb{R}^{n}),i=1,\ldots,d$ uniquely such that
\begin{equation}
\label{xi-martingale-repre}\xi= \sum_{i=1}^{d}\int_{0}^{T} \varphi
_{i}(s)dW_{i}(s).
\end{equation}
Since $L_{\mathbb{F}}^{2}([0,T];\mathbb{R}^{n})$ is complete and $W_{i}$ is
independent of $W_{j}$ when $1 \leq i \neq j \leq d$, we have for any
$\varphi(\cdot), \psi(\cdot)\in L_{\mathbb{F}}^{2}([0,T];\mathbb{R})$,
\begin{equation}
\label{inner-product-zero}\mathbb{E}\left[  \int_{0}^{T}\varphi(s)dW_{i}%
(s)\cdot\int_{0}^{T}\psi(s)dW_{j}(s)\right]  =0,1\leq i\neq j\leq d.
\end{equation}
Then it follows from (\ref{xi-martingale-repre}) and (\ref{inner-product-zero}%
) that $L_{0}^{2}(\mathcal{F}_{T})$ admits an orthogonal direct sum
\[
L_{0}^{2}(\mathcal{F}_{T})=\mathcal{L}_{1} \oplus\cdots\oplus\mathcal{L}_{d}
\]
such that
\begin{equation}
\label{eq-orthogonal-direct-sum}\mathcal{L}_{i}=\left\{  \int_{0}^{T}%
\varphi(s)dW_{i}(s):\varphi(\cdot)\in L_{\mathbb{F}}^{2}([0,T];\mathbb{R}%
^{n})\right\}  ,i=1,\ldots d.
\end{equation}
On the other hand, for $i=1,\ldots,d$, applying the martingale representation
theorem to any $\xi\in L_{0}^{2}(\mathcal{F}_{T}^{i})$ yields
\begin{equation}
\label{var-basis-space-Li}L_{0}^{2}(\mathcal{F}_{T}^{i})=\left\{  \int_{0}%
^{T}\varphi(s)dW_{i}(s):\varphi(\cdot)\in L_{\mathbb{F}^{i}}^{2}([0,T];\mathbb{R}%
^{n})\right\}  ,i=1,\ldots d.
\end{equation}
Obviously $L_{0}^{2}(\mathcal{F}_{T}^{i}) \subset\mathcal{L}_{i}, i=1,\ldots
d$. Due to Proposition \ref{prop-suffi-decompose}, we can construct a variance
decomposition $\{\mathrm{D}_{i},i=1,\ldots,m\}$ on $L^{2}(\mathcal{F}_{T})$
such that for any $\xi\in L^{2}(\mathcal{F}_{T})$ with the orthogonal
decomposition
\[
\xi= \mathbb{E}[\xi] + \sum_{i=1}^{d}\int_{0}^{T} \varphi_{i}(s)dW_{i}%
(s),\ \ \varphi_{i}(\cdot)\in L_{\mathbb{F}}^{2}([0,T];\mathbb{R}^{n}),i=1,\ldots,d,
\]
we have
\[
\mathrm{D}_{i}[\xi]=\mathbb{E}\left[  \int_{0}^{T}\left\vert \varphi
_{i}(s)\right\vert ^{2}ds\right]  ,\text{ }i=1,\ldots d
\]
and $\mathrm{Var}[\xi]=\sum_{i=1}^{d}\mathrm{D}_{i}[\xi]$.
\end{example}

With the help of Proposition \ref{prop-suffi-decompose} and Example \ref{example-2}, 
we can obtain the following mean-variance representation of the asymmetric risk-sensitive criterion introduced in Definition \ref{def-asymmetric-criterion}.

\begin{theorem}
\label{thm-mean-var-repre}  Under Assumption
\ref{Taylor-terminal-condition}, for any $\xi \in \mathrm{Dom}(\mathcal{E}_{\gamma_{1},\gamma_{2}})$ we have
\begin{equation}
\mathcal{E}_{\gamma_{1},\gamma_{2}}[\xi]=\mathbb{E}[\xi]+\gamma_{1}%
\mathrm{D}_{1}[\xi]+\gamma_{2}\mathrm{D}_{2}[\xi]+o\left(  \sqrt{\gamma
_{1}^{2}+\gamma_{2}^{2}}\right)  ,
\label{Taylor-expan-interpret-g-expectation}%
\end{equation}
where $\{D_{i},i=1,2\}$ is a variance decomposition on $L^{2}(\mathcal{F}_{T};\mathbb{R})$  such that
$\mathrm{Var}[\xi] = \mathrm{D}_{1}[\xi] + \mathrm{D}_{2}[\xi]$.
\end{theorem}

\begin{proof}
For any $\xi \in \mathrm{Dom}(\mathcal{E}_{\gamma_{1},\gamma_{2}})$, it follows from (\ref{quadratic-g-expectation}) that
$
\mathcal{E}_{\gamma_{1},\gamma_{2}}[\xi]=Y^{\gamma_{1},\gamma_{2}}(0).
$
On the one hand, according to (\ref{Taylor-opti-eq}) and (\ref{Taylor-expan-2}), we have
\[
\bar{Y}(0)=\mathbb{E}[\xi],\text{ }Y_{1}(0)=\mathbb{E}\left[  \int%
_{0}^{T}\bar{Z}_{1}^{2}(t)dt\right]  ,\text{ }Y_{2}(0)=\mathbb{E}%
\left[  \int_{0}^{T}\bar{Z}_{2}^{2}(t)dt\right]  ,
\]
and
\[
\xi=\mathbb{E}[\xi]+\int_{0}^{T}\bar{Z}_{1}(t)dW_{1}(t)+\int_{0}^{T}\bar{Z}%
_{2}(t)dW_{2}(t).
\]
On the other hand, 
from Example \ref{example-2} and the uniqueness of the martingale representation of $\xi$, we have
\[
\mathrm{D}_{1}[\xi]=\mathbb{E}\left[  \int_{0}^{T}\bar{Z}_{1}^{2}(t)dt\right]  ,\text{
}\mathrm{D}_{2}[\xi]=\mathbb{E}\left[  \int_{0}^{T}\bar{Z}_{2}^{2}(t)dt\right]  .
\]
Combining the above relationships with (\ref{Taylor-expan-express}) yields (\ref{Taylor-expan-interpret-g-expectation}).
\end{proof}

We interpret (\ref{Taylor-expan-interpret-g-expectation}) from the perspective in finance. As it is mentioned in \cite{Hu-Ma-Peng-Yao}, the left hand side of (\ref{Taylor-expan-interpret-g-expectation}),
$\mathcal{E}_{\gamma_{1},\gamma_{2}}[\xi]$, can be understood as a convex risk
measure about the derivative $\xi$ (maybe some future or some option contract)
based on the underlying asset $X$ adapted to the filtration $\mathbb{F}$
generated by $(W_{1},W_{2})$. The total risk measure of $\xi$ is
decomposed into three main parts---the right hand side of
(\ref{Taylor-expan-interpret-g-expectation}). For a decision maker, as
$\gamma_{1} \neq\gamma_{2}$ represents her asymmetric risk-sensitive attitudes
toward two different risk sources $W_{1}$ and $W_{2}$, in her criterion she
needs to distinguish the risks stemming from $W_{i},i=1,2$ so that
$\mathrm{Var}[\xi]$ is decomposed into $\mathrm{D}_{1}[\xi]$ and
$\mathrm{D}_{2}[\xi]$, and for $i=1,2$ she weights $\mathrm{D}_{i}[\xi]$ with
$\gamma_{i}$. 

\begin{corollary}
If $\gamma_{1} = \gamma_{2} = \frac{\theta}{2}>0$ then
(\ref{Taylor-expan-interpret-g-expectation}) becomes
\[
\mathcal{E}_{\theta}[\xi]= \mathbb{E}[\xi]+\frac{\theta}{2} \mathrm{Var}%
[\xi]+o\left( \theta\right)  , \ \forall \xi \in \mathrm{Dom}(\mathcal{E}_{\theta}),
\]
which is in accordance with the mean-variance representation as
(\ref{Taylor-expan-sym}) in symmetric risk-sensitive problem.
\end{corollary}

\section{Asymmetric risk-sensitive control problems}

Based on Theorem \ref{thm-mean-var-repre}, it is reasonable and suitable to employ 
(\ref{intro-obje-recursive}) with (\ref{intro-FBSDE-state-eq}) to
describe asymmetric risk-sensitive control problems, when $\Gamma$ are only strictly positive definite.

\subsection{Problem formulation}

We formulate the asymmetric risk-sensitive stochastic control problems as
follows. Consider the control system
\begin{equation}
\left\{
\begin{array}
[c]{rl}%
dX(t)= & b\left(  t,X(t),u(t)\right)  dt+\sigma(t,X(t),u(t))dW(t),\\
dY(t)= & -[Z^{\intercal}(t)\Gamma Z(t)+f(t,X(t),u(t))]dt+Z^{\intercal
}(t)dW(t),\\
X(0)= & x_{0},\ Y(T)=\Phi(X(T)),
\end{array}
\right.  \label{state-eqy0}%
\end{equation}
where $x_{0} \in\mathbb{R}^{n}$, $U\subset\mathbb{R}^{k}$ is a non-empty set, the $U$-valued process $u(\cdot)$ is the control process that will be defined later, and the coefficients
{
\small \[
b:[0,T]\times\mathbb{R}^{n}\times\mathbb{R}^{k} \rightarrow\mathbb{R}^{n},
\sigma:[0,T]\times\mathbb{R}^{n}\times\mathbb{R}^{k} \rightarrow
\mathbb{R}^{n\times d}, g: [0,T]\times\mathbb{R}^{n}\times\mathbb{R}%
^{k}\rightarrow\mathbb{R}, \Phi:\mathbb{R}^{n} \rightarrow\mathbb{R}
\]
}are measurable functions. $\Gamma\in\mathbb{S}^{d \times d}$ is strictly positive
definite. In (\ref{state-eqy0}), an $\mathbb{F}$-progressively measurable process $u(\cdot)$ is called an admissible control if the SDE admits a unique solution $X(\cdot)$ and the BSDE admits a minimal solution 
\footnote[1]{The general definition of the minimal solution of a quadratic BSDE can be found in \cite{Briand-H-06}. For the convenience of readers, here we only make the explanation about the minimal solution for 
the quadratic BSDE in (\ref{state-eqy0}).
A solution $(Y(\cdot),Z(\cdot))$ is said to be minimal if $\mathbb{P}$-a.s. for each $t \in [0,T]$, $Y(t) \leq Y^{\prime}(t)$ whenever $\left(Y^{\prime}(\cdot),Z^{\prime}(\cdot)\right)$ 
is another solution. Moreover, $(Y(\cdot),Z(\cdot))$ is said to be minimal in some space $\mathcal{B}$
if it belongs to this space and the previous property holds true as soon as $\left(Y^{\prime}(\cdot),Z^{\prime}(\cdot)\right) \in \mathcal{B}$.
Particularly, the minimal solution will become the unique solution while the BSDE in (\ref{state-eqy0}) is well posed.}
$(Y(\cdot),Z(\cdot))$ in specific spaces. Denote by $\mathcal{U}[0,T]$ the set of all the admissible controls. 
We'd like to emphasize that $\mathcal{U}[0,T]$ will be more specific according to the different cases we studied in the following context.
The cost functional is defined by
\begin{equation}
J(u(\cdot)):= Y(0), \ u(\cdot) \in \mathcal{U}[0,T].
\label{recursive-criterion}
\end{equation}
The objective is to find $\bar{u}(\cdot) \in\mathcal{U}[0,T]$ (if it ever exists) such that
\begin{equation}
J(\bar{u}(\cdot))=\inf\limits_{u(\cdot)\in\ \mathcal{U}[0,T]}J(u(\cdot)).
\label{obje-eq-y0}%
\end{equation}

Let us review some pioneering works about the risk-sensitive control
problems in details. Consider the risk-sensitive control problem with the cost
functional (\ref{intro-exp-objec}) with state equation
(\ref{intro-fward-state-eq}). The problem that minimizing
(\ref{intro-exp-objec}) subject to (\ref{intro-fward-state-eq}) was studied by
Lim and Zhou \cite{LimZhou05} in two cases. On the one hand, under the
assumption that $(\Phi,f)$ is uniformly bounded and the corresponding value
function defined from (\ref{intro-exp-objec}) is sufficiently smooth, they
obtained a new risk-sensitive MP for (\ref{intro-exp-objec})-(\ref{intro-fward-state-eq}). On the other hand, when $(\Phi,f)$ is no
longer bounded, the authors studied the risk-sensitive LQ
problems by taking
\begin{equation}
\label{prelim-LQ-coeff}%
\begin{array}
[c]{rl}%
b(t,x,u)= & A(t)x+B(t)u,\\
\sigma(t,x,u)= & \Sigma(t),\\
\Phi(x)= & \frac{1}{2}x^{\intercal}Hx,\\
f(t,x,u)= & \frac{1}{2}x^{\intercal}M(t)x+\frac{1}{2}u^{\intercal}N(t)u,
\end{array}
\end{equation}
where $A$, $B$, $\Sigma$, $H$, $M$, $N$ are respectively matrices or
matrix-valued, deterministic function on $[0,T]$ in suitable sizes. The
authors obtained the optimal control in the feedback form by verifying
sufficient conditions for optimality. After that, using neither stochastic MP nor DPP, a completion-of-square
approach is adopted by Duncan in \cite{Duncan} to solve this kind of problems,
so one can get rid of the assumption that the value function is sufficiently
smooth in, as called in \cite{Duncan}, the linear-exponential-quadratic Gaussian case.
Based on this review, in our new formulation (\ref{state-eqy0})-(\ref{obje-eq-y0}), we will primarily focus on bounded $(\Phi,f)$ and (\ref{prelim-LQ-coeff})
during the upcoming investigation.

\subsection{Asymmetric risk-sensitive control under bounded conditions} 

In our problem formulation, if $(\Phi,f)$ is uniformly bounded in
(\ref{state-eqy0}), then it can be addressed by applying the results in our
earlier work \cite{Hu-Ji-Xu22}. We adopt the spike variation approach to
obtain a global stochastic MP for the optimality of
(\ref{state-eqy0})-(\ref{obje-eq-y0}). For $\psi=b$, $\sigma$, $f$, $\Phi$,
denote $\psi(t)=\psi(t,\bar{X}(t),\bar{u}(t)),\psi_{x}(t)=\psi_{x}(t,\bar
{X}(t),\bar{u}(t))$, $t\in\lbrack0,T]$.

\begin{assumption}
\label{assum-1} (i) $b$, $\sigma$ are twice continuously differentiable with
respect to $x$. The derivatives $b_{x}$, $b_{xx}$, $\sigma_{x}$, $\sigma_{xx}$
are continuous in $(x,u)$ and uniformly bounded. $b,\sigma$ are bounded by
$C(1+|x|+|u|)$;

(ii) $f$, $\Phi$ are twice continuously differentiable with respect to $x$.
The derivatives $f_{x},f_{xx}$ are continuous in $(x,u)$; $f$, $\Phi$, $f_{x}%
$, $\Phi_{x}$, $f_{xx}$, $\Phi_{xx}$ are bounded.
\end{assumption}

We further assume the set of admissible controls
\[
\mathcal{U}_{BD}[0,T]=\{u:[0,T]\times\Omega\rightarrow U|\sup\limits_{0\leq t\leq
T}\mathbb{E}[|u(t)|^{p}]<+\infty, \text{ }\forall p>0 \}.
\]

Under Assumption \ref{assum-1}, it follows from Theorem 2.3 in \cite{Hu-Ji-Xu22} that for any $u(\cdot) \in \mathcal{U}_{BD}[0,T]$ (\ref{state-eqy0}) admits a unique solution $(X(\cdot),Y(\cdot),Z(\cdot)) \in L_{\mathbb{F}}^{2}(\Omega;C([0,T],\mathbb{R}^{2}))\times L_{\mathbb{F}}^{\infty}(\Omega;C([0,T],\mathbb{R})) \times L_{\mathbb{F}}^{2}([0,T];\mathbb{R}^{d})$ 
such that the stochastic integral $\{\int_{0}^{t} Z(s) dW(s), $ $t \in [0,T]\}$ is a bounded mean oscillation martingale. Therefore (\ref{obje-eq-y0}) is well defined.
Furthermore, let $\bar{u}(\cdot)$ be an optimal control and $(\bar{X}(\cdot),$ $\bar{Y}(\cdot),\bar{Z}(\cdot))$ be the
corresponding optimal state trajectory. As a special case of  our previous work (see \cite{Hu-Ji-Xu22} for the details), the global stochastic maximum principle
for the optimal control problem (\ref{state-eqy0})-(\ref{obje-eq-y0}) holds (see \cite{HJXX}).

\begin{remark}
\label{thm-global-SMP} 
When $\Gamma=\frac{\theta}{2}\mathrm{I}_{d\times d}$, under the same assumption in \cite{LimZhou05} that the value function is sufficiently smooth, the obtained stochastic maximum principle degenerates into the MP in \cite{LimZhou05}. Therefore,
without the smooth assumption imposed on the value function, it is more
convenient and straightforward to use (\ref{intro-obje-recursive}%
) and (\ref{intro-FBSDE-state-eq}) to formulate the risk-sensitive control problem
in \cite{LimZhou05} so that one does not have to introduce the auxiliary state
equation and use the logarithmic transformation.

\end{remark}

%

\begin{corollary}
\label{cor-local-SMP} Assume the coefficients are differentiable with respect
to $u$ and the control domain $U\subseteq\mathbb{R}^{k}$ is a convex set. Then
the foloowing stochastic maximum principle holds%
\begin{equation}
\label{local-SMP}H_{u}\mathcal{(}t,\bar{X}(t),\bar{Y}(t),\bar{Z}%
(t),u,p(t),q(t))|_{u=\bar{u}(t)}(u-\bar{u}(t))\geq0,\forall u\in U\text{,
}dt\otimes d\mathbb{P}\text{-a.e.},
\end{equation}
where%
\begin{equation}
H(t,x,z,u,p,q)=p^{\intercal}b(t,x,u)+\mathrm{tr}\left\{
q^{\intercal}\sigma(t,x,u)\right\}  +2p^{\intercal}\sigma(t,x,u)\Gamma
z+g(t,x,u), \label{def-H-u}%
\end{equation}
 $\left(  p(\cdot),q(\cdot)\right)  $, satisfies the adjoint equation
\begin{equation}
\left\{
\begin{array}
[c]{rl}%
dp(t)= & -\left\{  f_{x}(t)+2\sum\limits_{i=1}^{d}\left(  \Gamma\bar
{Z}(t)\right)  _{i}\left[  \sigma_{i,x}^{\intercal}(t)p(t)+q_{i}(t)\right]
\right. \\
& +\left.  b_{x}^{\intercal}(t)p(t)+\sum\limits_{i=1}^{d}\sigma_{i,x}%
^{\intercal}(t)q_{i}(t)\right\}  dt+\sum\limits_{i=1}^{d}q_{i}(t)dW_{i}(t),
\ \ t\in[0,T],\\
p(T)= & \Phi_{x}(\bar{X}(T))
\end{array}
\right.  \label{eq-p}%
\end{equation}
 $\{\sigma_{i,x}\}_{i=1,\ldots,d}$ is the Jacobian
matrix of the $i$th column of $\sigma$.
\end{corollary}

We shall utilize Corollary \ref{cor-local-SMP} to informally derive the candidate optimal control in the following subsection.

\subsection{Asymmetric LQ risk-sensitive control problems}

In this subsection, we consider a kind of unbounded $(\Phi,f)$ in
(\ref{state-eqy0}). Similar to \cite{LimZhou05}, we are interested in the
asymmetric LQ risk-sensitive control problems where $(\Phi,f)$ possesses
the forms in (\ref{prelim-LQ-coeff}). Consider the following stochastic
control system:
\begin{equation}
\left\{
\begin{array}
[c]{rl}%
dX(t)= & [A(t)X(t)+B(t)u(t)]dt+\Sigma(t)dW(t),\\
dY(t)= & -[Z^{\intercal}(t)\Gamma Z(t)+\frac{1}{2}X^{\intercal}%
(t)M(t)X(t)+\frac{1}{2}u^{\intercal}(t)N(t)u(t)]dt\\
& +Z^{\intercal}(t)dW(t),\text{ \ }t\in\lbrack0,T],\\
X(0)= & x_{0},\ Y(T)=\frac{1}{2}X^{\intercal}(T)HX(T),
\end{array}
\right.  \label{state-eq-yr}%
\end{equation}
where  $A(\cdot)\in L^{\infty}([0,T];\mathbb{R}^{n\times n})$, $B(\cdot)\in
L^{\infty}([0,T];\mathbb{R}^{n\times k})$, $\Sigma(\cdot)\in L^{2}([0,T];\mathbb{R}^{n\times d})$, $M(\cdot)\in L^{\infty}([0,T];\mathbb{S}^{n\times n})$, $N(\cdot)\in L^{\infty}([0,T];\mathbb{S}^{k\times k})$ are deterministic matrix-valued functions; $H\in\mathbb{S}^{n\times n}$ and $H\geq0$; $M(t)\geq0$, $N(t)\geq\delta\mathrm{I}_{k\times k}$ for some $\delta>0$ and all $t\in\lbrack0,T]$; $(\Sigma\Sigma^{\intercal})(t)>0$, $t\in\lbrack0,T]$. Set $\Delta=\int_{0}^{T}(\Sigma\Sigma^{\intercal})(s)ds$ and denote by two positive numbers $\gamma_{\mathrm{max}}$, $\gamma_{\mathrm{min}}$ the maximal, minimal eigenvalues of $\Gamma$ respectively.

For any $u(\cdot)\in L_{\mathbb{F}}^{2}([0,T];\mathbb{R}^{k})$, under the above conditions, the SDE in (\ref{state-eq-yr}) admits a unique solution $X(\cdot) \in L_{\mathbb{F}}^{2}(\Omega;C([0,T],\mathbb{R}^{n}))$ according to the standard theory. To describe our goal under asymmetric LQ setting, we introduce the set of all admissible controls.

\begin{definition}
\label{def-admmisible-lq} Denote by $\mathcal{U}_{LQ}[0,T]$ the admissible control set that is given by
\[%
\begin{array}
[c]{rl}%
\mathcal{U}_{LQ}[0,T]:= & \left\{  u(\cdot)\in L_{\mathbb{F}}^{2}%
([0,T];\mathbb{R}^{k}): \text{the BSDE in (\ref{state-eq-yr}) admits a minimal }\right.  \\
& \left.\text{ solution} (Y(\cdot),Z(\cdot)) \in L_{\mathbb{F}}^{2}(\Omega;C([0,T],\mathbb{R}))\times L_{\mathbb{F}}^{2}([0,T];\mathbb{R}^{d}) \right\}.
\end{array}
\]
\end{definition}

\begin{remark}
	By Proposition 4 in \cite{Briand-H-06}, for given $u(\cdot)\in L_{\mathbb{F}}^{2}([0,T];\mathbb{R}^{k})$ if
	the BSDE in (\ref{state-eq-yr}) has a solution in $L_{\mathbb{F}}^{2}(\Omega;C([0,T],\mathbb{R}))\times L_{\mathbb{F}}^{2}([0,T];\mathbb{R}^{d})$, 
	then it admits a minimal solution in $L_{\mathbb{F}}^{2}(\Omega;C([0,T],\mathbb{R}))\times L_{\mathbb{F}}^{2}([0,T];\mathbb{R}^{d})$. 
\end{remark}

The cost functional is defined by
\begin{equation}
J(u(\cdot)):=Y(0), \ u(\cdot) \in \mathcal{U}_{LQ}[0,T].
\label{obje-eq-yr}
\end{equation}
The objective is to find $\bar{u}(\cdot) \in \mathcal{U}_{LQ}[0,T]$ (if it ever exists) such that
\[
J(\bar{u}(\cdot))=\inf\limits_{u(\cdot)\in\ \mathcal{U}_{LQ}[0,T]}J(u(\cdot)).
\]

To determine the optimal feedback control like the classical LQ
case, we resort to the aid of Corollary \ref{cor-local-SMP}.
It should be pointed out that the following derivation based on the Corollary \ref{cor-local-SMP} is
just informal and heuristic because $\frac{1}{2}X^{\intercal}(T)HX(T)$ and $\frac{1}
{2}X^{\intercal}(t)M(t)X(t)+\frac{1}{2}u^{\intercal}(t)N(t)u(t)$ are not bounded.
The first-order adjoint equation (\ref{eq-p}) now becomes
\begin{equation}
\left\{
\begin{array}
[c]{rl}%
dp(t)= & -[A^{\intercal}(t)p(t)+2q(t)\Gamma\bar{Z}(t)+M(t)\bar{X}(t)]dt+q(t)dW(t),\\
p(T)= & H\bar{X}(T).
\end{array}
\right.  \label{eq-lq-adjoint}%
\end{equation}
Let $\left(  p(\cdot), q(\cdot) \right)  \in L_{\mathbb{F}}^{2}(\Omega;C([0,T],\mathbb{R}^{n}))
\times L_{\mathbb{F}}^{2}([0,T];\mathbb{R}^{n\times d})$ be the unique solution to (\ref{eq-lq-adjoint}).
Since the control set $\mathbb{R}^{k}$ is convex, by the
minimum condition in Corollary \ref{cor-local-SMP}, the candidate optimal control
possesses the following form
\begin{equation}
\bar{u}(t)=-N^{-1}(t)B^{\intercal}(t)p(t).
\label{candidate-opti-control}%
\end{equation}
We assume the relationship
\begin{equation}
	\label{barZ-Sigam-p}
	\bar{Z}(t)=\Sigma^{\intercal}(t)p(t)
\end{equation}
and conjecture that $p(\cdot)$ and $\bar{X}(\cdot)$ are related by
\begin{equation}
p(t)=P(t)\bar{X}(t)
\label{re-conje}%
\end{equation}
with $P(\cdot)\in C^{1}([0,T],\mathbb{S}^{n\times n})$. Applying It\^{o}'s formula to (\ref{re-conje}) yields
\begin{equation}%
	dp(t) = \left[  P^{(1)} (t) \bar{X}(t)+P(t)A(t)\bar{X}(t)+P(t)B(t)\bar{u}(t)\right]
	dt+P(t)\Sigma(t)dW(t).
\label{eq-p-P-relation}
\end{equation}
Due to relationships (\ref{candidate-opti-control}), (\ref{barZ-Sigam-p}), and (\ref{re-conje}), comparing (\ref{eq-p-P-relation}) with (\ref{eq-lq-adjoint}), we
obtain that $P(\cdot)$ should satisfied the Riccati differential equation
\begin{equation}
\left\{
\begin{array}
[c]{rl}%
dP(t)= & -\left[  A^{\intercal}(t)P(t)+P(t)A(t)+M(t)\right. \\
& +\left.  P(t)\left(  2\Sigma(t)\Gamma\Sigma^{\intercal}(t)-B(t)N^{-1}(t)B^{\intercal}(t)\right)  P(t)\right]  dt,\text{ \ }%
t\in\lbrack0,T],\\
P(T)= & H.
\end{array}
\right.  \label{eq-P-ode}%
\end{equation}

\begin{lemma}
\label{lem-Riccati-wellpose-comparison} Assume $\Sigma(t)$, $N(t)$ are both
continuous in $t$ and
\begin{equation}
\label{Riccati-P-wellpose-cond}2\Sigma(t)\Gamma\Sigma^{\intercal}(t) -
B(t)N(t)^{-1}B^{\intercal}(t) < 0, \ \forall t \in[0,T].
\end{equation}
Then (\ref{eq-P-ode}) admits a unique solution $P(\cdot) \in C([0,T],\mathbb{S}%
^{n\times n})$ such that $P(t) \geq0$ for all $t \in[0,T]$ and $\left\Vert P
\right\Vert _{\infty} \leq B_{P}$, where $B_{P}:=e^{2\left\Vert A\right\Vert _{\infty}T}\left(  \left\Vert
H\right\Vert _{\infty}+\left\Vert M\right\Vert _{\infty}T\right)  $.
\end{lemma}

\begin{proof}
According to the continuity of $\Sigma(\cdot)$ and $N(\cdot)$, if (\ref{Riccati-P-wellpose-cond}) holds then (\ref{eq-P-ode}) admits a unique solution by the classical Riccati theory (We refer the readers to \cite{Wonham-1968} for more details). To prove the last claim, let $\tilde{P}(\cdot) \in C([0,T],\mathbb{S}^{n\times n})$ be the unique solution to the linear ordinary differential equation
\begin{equation}
\left\{
\begin{array}
[c]{rl}%
d\tilde{P}(t)= & -\left[  A^{\intercal}(t)\tilde{P}(t)+\tilde{P}(t)A(t)+M(t)\right]  dt,\text{ \ }t\in\lbrack0,T],\\
\tilde{P}(T)= & H.
\end{array}
\right.
\label{eq-tilde-P-ode}%
\end{equation}
Thanks to Grownwall's lemma, one can easily deduce that $\left\Vert \tilde{P} \right\Vert_{\infty} \leq B_{P}$.
On the other hand, since $2\Sigma(t)\Gamma \Sigma^{\intercal}(t) - B(t)N(t)^{-1}B^{\intercal}(t) < 0$ for all $t \in [0,T]$ and $H \geq 0$, by Theorem 2.2 in \cite{Freiling-Jank} we have $0 \leq P(t) \leq \tilde{P}(t), \forall t \in [0,T]$. Then from the definition of the Frobenius norm for the real matrices we obtain $\left\Vert P \right\Vert_{\infty} \leq \left\Vert \tilde{P} \right\Vert_{\infty} \leq B_{P}$.
\end{proof}

\begin{remark}
When $\Gamma=\frac{\theta}{2}\mathrm{I}_{d \times d}$ and the coefficients in (\ref{state-eq-yr})
are all constant matrices, (\ref{eq-P-ode}) is identical to the one introduced in \cite{Duncan}.
\end{remark}

The following main result indicates that if the Riccati differential equation (\ref{eq-P-ode}) admits a unique solution, then $\mathcal{U}_{LQ}[0,T]$ is not empty and there exists an admissible control with feedback type optimizing the problem (\ref{state-eq-yr})-(\ref{obje-eq-yr}).

\begin{theorem}
\label{thm-lq-optimal} If $P(\cdot) \in C([0,T],\mathbb{S}^{n\times n})$
uniquely solves (\ref{eq-P-ode}), then the feedback control
\begin{equation}
	\bar{u}(t)=-N^{-1}(t)B^{\intercal}(t)P(t)\bar{X}(t) \ \ t\in[0,T]
\label{eq-opt-con-lq}%
	\end{equation}
belongs to $\mathcal{U}_{LQ}[0,T]$ and is optimal for the problem (\ref{state-eq-yr})-(\ref{obje-eq-yr}). The optimal value of the objective function is
\[
J(\bar{u}(\cdot))=\frac{1}{2}x_{0}^{\intercal}P(0)x_{0}+\frac{1}{2}\int_{0}%
^{T}\mathrm{tr}\left\{  P(t)\left(  \Sigma\Sigma^{\intercal}\right)
(t)\right\}  dt.
\]
\end{theorem}


\begin{proof}
Plugging (\ref{eq-opt-con-lq}) into the SDE in (\ref{state-eq-yr}), it admits a unique solution $\bar{X}(\cdot) \in \bigcap_{p>1} L_{\mathbb{F}}^{p}(\Omega;C([0,T],\mathbb{R}^{n}))$ as it is Gaussian, which implies $\bar{u}(\cdot) \in L_{\mathbb{F}}^{2}([0,T];\mathbb{R}^{k})$.
To prove  $\bar{u}(\cdot) \in \mathcal{U}_{LQ}[0,T]$, we have to find the minimal solution in $L_{\mathbb{F}}^{2}(\Omega;C([0,T],\mathbb{R}))$ $\times L_{\mathbb{F}}^{2}([0,T];\mathbb{R}^{d})$ corresponding to $\bar{u}(\cdot)$. Applying It\^{o}'s formula to $\bar{X}^{\intercal}(t)P(t)\bar{X}(t)$ yields
\begin{equation}%
\begin{array}
[c]{rl}
& \frac{1}{2}\bar{X}^{\intercal}(t)P(t)\bar{X}(t)+\frac{1}{2} \int_{t}^{T} \mathrm{tr}\left\{  P(s)\left(  \Sigma\Sigma^{\intercal}\right)  (s)\right\}  ds\\
= & \frac{1}{2}\bar{X}^{\intercal}(T)H\bar{X}(T)+\int_{t}^{T}\left[  \bar
{X}^{\intercal}(s)P(s)\Sigma(s)\Gamma\Sigma^{\intercal}(s)P(s)\bar
{X}(s)\right.  \\
& +\left.  \frac{1}{2}\bar{X}^{\intercal}(s)M(s)\bar{X}(s)+\frac{1}{2}\bar
{u}^{\intercal}(s)N(s)\bar{u}(s)\right]  ds-\int_{t}^{T}\bar{X}^{\intercal
}(s)P(s)\Sigma(s)dW(s),
\end{array}
\label{P-barX2-psi}%
\end{equation}
which implies that
\begin{equation}
\label{barY-barZ-barX-relation}
\begin{array}
[c]{rl}%
\bar{Y}(t):= & \frac{1}{2}\bar{X}^{\intercal}(t)P(t)\bar{X}(t)+\frac{1}{2} \int_{t}^{T} \mathrm{tr}\left\{  P(s)\left(  \Sigma\Sigma^{\intercal}\right)  (s)\right\}  ds,\\
\bar{Z}(t):= & \Sigma^{\intercal}(t)P(t)\bar{X}(t)
\end{array}
\end{equation}
solves the BSDE in (\ref{state-eq-yr}) when $u(\cdot)=\bar{u}(\cdot)$, and it belongs to $L_{\mathbb{F}}^{2}(\Omega;C([0,T],\mathbb{R}))\times L_{\mathbb{F}}^{2}([0,T];\mathbb{R}^{d})$ as $\bar{X}(\cdot) \in L_{\mathbb{F}}^{4}(\Omega;C([0,T],\mathbb{R}^{n}))$. We claim that $(\bar{Y}(\cdot),\bar{Z}(\cdot))$ is minimal. Actually, if $(\bar{Y}^{\prime}(\cdot),\bar{Z}^{\prime}(\cdot)) \in L_{\mathbb{F}}^{2}(\Omega;C([0,T],\mathbb{R}))\times L_{\mathbb{F}}^{2}([0,T];\mathbb{R}^{d})$ is another solution corresponding to $\bar{u}(\cdot)$, then we have for any $t \in [0,T]$,
\[
\bar{Y}^{\prime}(t)-\bar{Y}(t) \geq \int_{t}^{T}2\left[  \bar{Z}^{\prime}
(s)-\bar{Z}(s)\right]  ^{\intercal}\Gamma\bar{Z}(s)ds-\int_{t}^{T}\left[
\bar{Z}^{\prime}(s)-\bar{Z}(s)\right]  ^{\intercal}dW(s).
\]
Consider a new probability measure $\bar{\mathbb{P}}$ defined by the stochastic exponential
\begin{equation}
d\bar{\mathbb{P}}=\exp\left\{  2\int_{0}^{T}\bar{Z}^{\intercal}(s)\Gamma
dW(s)-\int_{0}^{T}\left\vert \Gamma\bar{Z}(s)\right\vert ^{2}ds\right\}
d\mathbb{P}.\label{LQ-R-N-derivative}%
\end{equation}
Because $\bar{X}(\cdot)$ is a Gaussian process, the above stochastic exponential is a Radon-Nikodym derivative that integrates to one, according to (\ref{barY-barZ-barX-relation}) and the argument in \cite{Skorokhod} that the Girsanov exponential with Gaussian integrand is an exponential martingale.
According to Girsanov's theorem, we deduce
\begin{equation}
\label{difference-barYp-barY}
\bar{Y}^{\prime}(t)-\bar{Y}(t) \geq -\int_{t}^{T}\left[\bar{Z}^{\prime}(s)-\bar{Z}(s)\right]  ^{\intercal}d\bar{W}(s),
\end{equation}
where
\[
\bar{W}(t)=W(t)-2\int_{0}^{t}\Gamma\bar{Z}(s)ds,\ \ t \in [0,T]
\]
is a $d$ dimensional Brownian motion under $\bar{\mathbb{P}}$. Denote by $\mathbb{E}_{\bar{\mathbb{P}}}\left[  \cdot\right]  $ the mathematical expectation corresponding to $\bar{\mathbb{P}}$. It will be proved later that the right-hand side of inequality (\ref{difference-barYp-barY}) is a true martingale. Taking the conditional expectation $\mathbb{E}_{\bar{\mathbb{P}}} \left[ \cdot \mid \mathcal{F}_{t} \right]$ in both sides of (\ref{difference-barYp-barY}) yields $Y^{\prime}(t) \geq Y(t)$, $\bar{\mathbb{P}}$-a.s. (of course, $\mathbb{P}$-a.s.) Therefore $\bar{u}(\cdot) \in \mathcal{U}_{LQ}[0,T]$ and
\begin{equation}
\label{barY-eq-Jbaru}
J(\bar{u}(\cdot))=\bar{Y}(0).
\end{equation}

Now we prove the optimality of $\bar{u}(\cdot)$. For any $u(\cdot) \in \mathcal{U}_{LQ}[0,T]$, if $(Y(\cdot),Z(\cdot)) \in L_{\mathbb{F}}^{2}(\Omega;C([0,T],\mathbb{R}))\times L_{\mathbb{F}}^{2}([0,T];\mathbb{R}^{d})$ is the minimal solution to the BSDE in (\ref{state-eq-yr}) corresponding to $u(\cdot)$, then the convexity of the quadratic function leads to
\[%
\begin{array}
[c]{rl}
& Y(0)-\bar{Y}(0)\\
\geq & \bar{X}^{\intercal}(T)H\left[  X(T)-\bar{X}(T)\right]  +\int_{0}%
^{T}\left\{  2\left[  Z(s)-\bar{Z}(s)\right]  ^{\intercal}\Gamma\bar
{Z}(s)\right.  \\
& +\left.  \bar{X}^{\intercal}(s)M(s)\left[  X(s)-\bar{X}(s)\right]  +\bar
{u}^{\intercal}(s)N(s)\left[  u(s)-\bar{u}(s)\right]  \right\}  ds\\
& -\int_{0}^{T}\left[  Z(s)-\bar{Z}(s)\right]  ^{\intercal}dW(s)
\end{array}
\]
and then applying It\^{o}'s formula to $X^{\intercal}(t)P(t)(X(t)-\bar{X}(t))$ over $[0,T]$ yields
\begin{equation}%
\begin{array}
[c]{rl}
& Y(0)-\bar{Y}(0)\\
\geq & \int_{0}^{T}\left\{  2\left[  Z(s)-\bar{Z}(s)-\Sigma^{\intercal
}(s)P(s)\left(  X(s)-\bar{X}(s)\right)  \right]  ^{\intercal}\Gamma\bar
{Z}(s)\right.  \\
& +\left.  \left[  \bar{X}^{\intercal}(s)P(s)B(s)+\bar{u}^{\intercal
}(s)N(s)\right]  \left[  u(s)-\bar{u}(s)\right]  \right\}  ds\\
& -\int_{0}^{T}\left[  Z(s)-\bar{Z}(s)-\Sigma^{\intercal}(s)P(s)\left(
X(s)-\bar{X}(s)\right)  \right]  ^{\intercal}dW(s).
\end{array}
\label{Ju-Jbaru-difference}%
\end{equation}
Noting
$$
\bar{X}^{\intercal}(s)P(s)B(s)+\bar{u}^{\intercal}(s)N(s) = 0,
$$
by Girsanov's theorem, (\ref{Ju-Jbaru-difference}) implies that
\begin{equation}
Y(0)-\bar{Y}(0)\geq-\int_{0}^{T}\left[  Z(s)-\bar{Z}(s)-\Sigma^{\intercal
}(s)P(s)\left(  X(s)-\bar{X}(s)\right)  \right]  ^{\intercal}d\bar
{W}(s),\label{defference-geq-local-martingale}%
\end{equation}
It will be proved later that the right-hand side of inequality (\ref{defference-geq-local-martingale}) is a true martingale. Taking $\mathbb{E}_{\bar{\mathbb{P}}} \left[ \cdot \right]$ in both sides of (\ref{defference-geq-local-martingale}), we have $Y(0)-\bar{Y}(0)\geq 0$, which means that
\begin{equation}
\label{LQ-optimality-claim-ineq}
J(u(\cdot))\geq J(\bar{u}(\cdot))
\end{equation}
due to (\ref{obje-eq-yr}) and (\ref{barY-eq-Jbaru}). The optimality of $\bar{u}(\cdot)$ follows from (\ref{LQ-optimality-claim-ineq}) and the arbitrariness of $u(\cdot)$ chosen from $\mathcal{U}_{LQ}[0,T]$.
Combining (\ref{barY-barZ-barX-relation}) and (\ref{barY-eq-Jbaru}) results in the optimal value
\begin{equation}
\label{barY-equals-Jbaru}
J(\bar{u}(\cdot))=\bar{Y}(0)=\frac{1}{2}x_{0}^{\intercal}P(0)x_{0}+\frac{1}{2} \int_{0}^{T} \mathrm{tr}\left\{  P(s)\left(  \Sigma\Sigma^{\intercal}\right)  (s)\right\}  ds.
\end{equation}

It remains to prove that the right-hand side of (\ref{difference-barYp-barY}) (resp. (\ref{defference-geq-local-martingale})) is true martingale under $\bar{\mathbb{P}}$ so we can take $\mathbb{E}_{\bar{\mathbb{P}}} \left[ \cdot \mid \mathcal{F}_{t}\right]$ (resp. $\mathbb{E}_{\bar{\mathbb{P}}} \left[ \cdot \right]$) to eliminate the stochastic integral. Since we have
\[
\mathbb{E}_{\mathbb{\bar{P}}}\left[  \left(\int_{0}^{T}\left\vert \bar{Z}(s) \right\vert ^{2}ds \right)^{\frac{p}{2}} \right] < +\infty, \ \forall p>1
\]
due to the fact that $\bar{X}(\cdot)$ is also Gaussian under $\mathbb{\bar{P}}$ and $\bar{Z}(s)=\Sigma^{\intercal}P(s)\bar{X}(s)$, it comes down to proving that if $(Y(\cdot),Z(\cdot)) \in L_{\mathbb{F}}^{2}(\Omega;C([0,T],\mathbb{R}))\times L_{\mathbb{F}}^{2}([0,T];\mathbb{R}^{d})$ is a solution of the BSDE in (\ref{state-eq-yr}) corresponding to a given $u(\cdot) \in L_{\mathbb{F}}^{2}([0,T];\mathbb{R}^{k})$, then $\{\int_{0}^{t} Z(s) dW(s), t \in[0,T]\}$ is a true martingale under $\bar{\mathbb{P}}$. To this end, for $m \geq 1$, let $\tau_{m}$ be the following stopping time
\[
\tau_{m}=\inf \left\{ t \geq 0: \int_{0}^{t} \left\vert Z(s) \right\vert^{2} \geq m \right\} \wedge T.
\]
As $Y(t) \geq 0$, $t \in [0,T]$, $\mathbb{P}$-a.s. (of course $\bar{\mathbb{P}}$-a.s.), then for each $m$ we have
\[%
\begin{array}
[c]{rl}%
Y(0)\geq & Y(t\wedge\tau_{m})+\int_{0}^{t\wedge\tau_{m}}Z^{\intercal}(s)\Gamma
Z(s)ds-\int_{0}^{t\wedge\tau_{m}}Z^{\intercal}(s)dW(s)\\
\geq & \int_{0}^{t\wedge\tau_{m}}\left[  Z^{\intercal
}(s)\Gamma Z(s)-2Z^{\intercal}(s)\Gamma\bar{Z}(s)\right]  ds-\int_{0}%
^{t\wedge\tau_{m}}Z^{\intercal}(s)d\bar{W}(s)\\
\geq & \int_{0}^{t\wedge\tau_{m}}\left[  \frac{\gamma_{\mathrm{min}}}{2}\left\vert
Z(s)\right\vert ^{2}-\frac{2\gamma_{\mathrm{max}}^{2}}{\gamma_{\mathrm{min}}}\left\vert \bar
{Z}(s)\right\vert ^{2}\right]  ds-\int_{0}^{t\wedge\tau_{m}}Z^{\intercal
}(s)d\bar{W}(s),
\end{array}
\]
which implies
\[
\begin{array}
[c]{rl}
\int_{0}^{\tau_{m}}\left\vert Z(s)\right\vert ^{2}ds\leq & \frac{2}{\gamma_{\mathrm{min}}}Y(0)+4\left(  \frac{\gamma_{\mathrm{max}}}{\gamma_{\mathrm{min}}}\right)  ^{2}\int_{0}^{T}\left\vert
\bar{Z}(s)\right\vert ^{2}ds\\
&+\frac{2}{\gamma_{\mathrm{min}}} \cdot \sup\limits_{t\in [0,T]}\left\vert \int%
_{0}^{t\wedge\tau_{m}}Z^{\intercal}(s)d\bar{W}(s)\right\vert .
\end{array}
\]
Taking $\mathbb{E}_{\bar{\mathbb{P}}} \left[ \cdot \right]$ in both sides of the last inequality and applying the Burkholder-Davis-Gundy inequality yield
\[
\begin{array}
[c]{rl}
\mathbb{E}_{\mathbb{\bar{P}}}\left[  \int_{0}^{\tau_{m}}\left\vert
Z(s)\right\vert ^{2}ds\right]  \leq &\frac{2}{\gamma_{\mathrm{min}}}Y(0)+4\left(  \frac
{\gamma_{\mathrm{max}}}{\gamma_{\mathrm{min}}}\right)  ^{2}\mathbb{E}_{\mathbb{\bar{P}}}\left[
\int_{0}^{T}\left\vert \bar{Z}(s)\right\vert ^{2}ds\right]  \\
&+\frac{6}{\gamma_{\mathrm{min}}}\mathbb{E}%
_{\mathbb{\bar{P}}}\left[  \left(  \int_{0}^{\tau_{m}}\left\vert
Z(s)\right\vert ^{2}ds\right)  ^{\frac{1}{2}}\right]  .
\end{array}
\]
Consequently, putting $a=\left(  \int_{0}^{\tau_{m}}\left\vert
Z(s)\right\vert ^{2}ds\right)  ^{\frac{1}{2}}$ and from the fundamental inequality $a \leq \frac{1}{2} (\frac{\gamma_{\mathrm{min}}}{6}a^{2}  +\frac{6}{\gamma_{\mathrm{min}}})$, we obtain
\[
\mathbb{E}_{\mathbb{\bar{P}}}\left[  \int_{0}^{\tau_{m}}\left\vert
Z(s)\right\vert ^{2}ds\right]  \leq\frac{4}{\gamma_{\mathrm{min}}}Y(0)+8\left(  \frac
{\gamma_{\mathrm{max}}}{\gamma_{\mathrm{min}}}\right)  ^{2}\mathbb{E}_{\mathbb{\bar{P}}}\left[
\int_{0}^{T}\left\vert \bar{Z}(s)\right\vert ^{2}ds\right]  +\frac{36}{\gamma_{\mathrm{min}}^{2}}<+\infty,
\]
which yields $\mathbb{E}_{\mathbb{\bar{P}}}\left[  \int_{0}^{T}\left\vert
Z(s)\right\vert ^{2}ds\right] < +\infty$ immediately from Fatou's lemma. Thus the stochastic integrals in the right sides of (\ref{difference-barYp-barY}) and (\ref{defference-geq-local-martingale}) are both true martingales under $\mathbb{\bar{P}}$. The proof is complete.
\end{proof}
	
As the end of this section, we emphasize that the admissible control set $\mathcal{U}_{LQ}[0,T]$ and the cost functional (\ref{obje-eq-yr}) provide a natural perspective to tackle with the LQ risk-sensitive control problem with identical risk-sensitive attitudes towards different risk sources. To illustrate this, putting $\Gamma=\frac{\theta}{2}\mathrm{I}_{d\times d}$ for some $\theta>0$, we find that a process $u(\cdot)\in L_{\mathbb{F}}^{2}([0,T];\mathbb{R}^{k})$ belongs to $\mathcal{U}_{LQ}[0,T]$ if and only if
\begin{equation}
\label{sym-LQ-admissible-set-condition}
\mathbb{E}\left[  \exp\left\{  \theta \left( \frac{1}{2} X^{\intercal
}(T)HX(T)+\frac{1}{2} \int_{0}^{T}g(t)dt\right)  \right\}  \right]  <+\infty,
\end{equation}
where $g(t)=X^{\intercal}(t)M(t)X(t)+u^{\intercal}(t)N(t)u(t),t\in [0,T]$. Actually, on the one hand, benefiting from the proof of Theorem 3.1 in \cite{Briand-Lepeltier-Martin} and the $L^{1}$-martingale representation theorem (please refer to Theorem 2.46 in \cite{Pardoux-book}), (\ref{sym-LQ-admissible-set-condition}) is sufficient to construct a solution $(Y(\cdot),Z(\cdot)) \in L_{\mathbb{F}}^{2}(\Omega;C([0,T],\mathbb{R}))\times L_{\mathbb{F}}^{2}([0,T];\mathbb{R}^{d})$ of the BSDE in (\ref{state-eq-yr}) such that
\begin{equation}
\label{sym-LQ-solution-risk-sensitive}
e^{\theta Y(0)} = \mathbb{E}\left[  \exp\left\{  \theta \left( \frac{1}{2} X^{\intercal
}(T)HX(T)+\frac{1}{2} \int_{0}^{T}g(t)dt\right)  \right\}  \right].
\end{equation}
On the other hand, according to Theorem 3.1 in \cite{Briand-Lepeltier-Martin}, (\ref{sym-LQ-admissible-set-condition}) is also necessary to guarantee the BSDE in (\ref{state-eq-yr}) admits at least one solution in $L_{\mathbb{F}}^{2}(\Omega;C([0,T],\mathbb{R}))\times L_{\mathbb{F}}^{2}([0,T];\mathbb{R}^{d})$ since
\begin{equation}
\label{sym-LQ-solution-dominate}
\mathbb{E}\left[  \exp\left\{  \theta \left( \frac{1}{2} X^{\intercal
}(T)HX(T)+\frac{1}{2} \int_{0}^{T}g(t)dt\right)  \right\}  \right] \leq e^{\theta Y^{\prime}(0)} <+\infty
\end{equation}
for any solution $(Y^{\prime}(\cdot),Z^{\prime}(\cdot)) \in L_{\mathbb{F}}^{2}(\Omega;C([0,T],\mathbb{R}))\times L_{\mathbb{F}}^{2}([0,T];\mathbb{R}^{d})$. It follows from (\ref{sym-LQ-solution-risk-sensitive}) and (\ref{sym-LQ-solution-dominate}) that (\ref{obje-eq-yr}) can be expressed by
\[
J(u(\cdot))=Y(0)=\frac{1}{\theta} \log \mathbb{E}\left[  \exp\left\{  \theta \left( \frac{1}{2} X^{\intercal
}(T)HX(T)+\frac{1}{2} \int_{0}^{T}g(t)dt\right)  \right\}  \right],
\]
which is nothing but the cost functional of the LQ risk-sensitive control with identical risk-sensitive attitudes towards different risk sources studied by Lim and Zhou \cite{LimZhou05}, and Duncan \cite{Duncan}. Resorting to Theorem \ref{thm-lq-optimal}, we can obtain the same feedback control. Thus (\ref{sym-LQ-admissible-set-condition}) completely characterizes $\mathcal{U}_{LQ}[0,T]$.

\section{An application to dynamic portfolio optimization}
As an application of risk-sensitive control in mathematical finance, the continuous time portfolio optimization problems with identical risk-sensitive attitudes towards different risk sources are well studied. Let $d=m+n$ with two positive integer $m,n$ and let $a\in\mathbb{R}^{m}$, $b\in\mathbb{R}^{n}$, and $A$, $B$, $\Lambda$, $\Sigma$ be respectively
$m\times n$, $n\times n$, $n\times d$, $m\times d$ constant matrices, and $r(t)$
be a nonnegative, deterministic function of $t$. If $\Gamma=\frac{\theta}{4}\mathrm{I}_{d\times d}$ for some $\theta>0$, then (\ref{state-eqy0})-(\ref{obje-eq-y0}) is closely related
to the portfolio optimization problem studied in \cite{Kuroda-Nagai} that the objective is to maximize the risk-sensitized expected growth rate up to time horizon $T$:
\begin{equation}
\label{portfolio-criterion}
I(u(\cdot)):=-\frac{2}{\theta}\log \mathbb{E}\left[ \exp\left\{ -\frac{\theta}{2} \log V(T) \right\} \right],
\end{equation}
where $V(\cdot)$ represents the investor's wealth process that is described by
\begin{equation}
\label{wealth-process-dynamic}
\frac{dV(t)}{V(t)}=r(t)dt + u^{\intercal}(t)\left(a + A\tilde{X}(t)-r(t) \mathbf{1}\right) dt + u^{\intercal}(t) \Sigma dW(t)
\end{equation}
with $\mathbf{1}:= (\overbrace{1,\ldots,1}^{n})$. Here the factor process $\tilde{X}$ satisfies the SDE
\begin{equation}
\left\{
\begin{array}
[c]{rl}%
d\tilde{X}(t)= & (b+B\tilde{X}(t))dt+\Lambda dW(t),\\
\tilde{X}(0)= & x_{0},
\end{array}
\right.
\label{eq-factor-process}
\end{equation}
which is interpreted as an exogenous macroeconomic, microeconomic or statistical process driving asset returns. 
Then for any $u(\cdot) \in \mathcal{U}[0,T]$ we have $I(u(\cdot))=-J(u(\cdot))$ if we put
\begin{equation}%
\begin{array}
[c]{rl}%
b(t,x,u)= & \left(
\begin{array}
[c]{cc}%
0 & \mathbf{0}_{1\times n}\\
\mathbf{0}_{n\times1} & B
\end{array}
\right)  \left(
\begin{array}
[c]{c}%
x_{1}\\
x_{2}%
\end{array}
\right)  +\left(
\begin{array}
[c]{c}%
0\\
b
\end{array}
\right)  ,\\
\sigma(t,x,u)= & \left(
\begin{array}
[c]{c}%
-u^{\intercal}\Sigma\\
\Lambda
\end{array}
\right)  ,\text{ \ }\Phi(x)=x_{1},\\
f(t,x,u)= & \frac{1}{2}u^{\intercal}\Sigma\Sigma^{\intercal}u-u^{\intercal
}(a+Ax_{2}-r(t)\mathbf{1})-r(t)
\end{array}
\label{prelim-log-coeff}%
\end{equation}
in (\ref{state-eqy0}) with $x = (x_{1},x_{2})
\in\mathbb{R} \times\mathbb{R}^{n}$, where $\mathbf{0}_{n\times1}^{\intercal}
= \mathbf{0}_{1\times n} = (\overbrace{0,\ldots,0}^{n})$. Therefore maximizing $I(\cdot)$ over $\mathcal{U}[0,T]$ is equivalent to minimizing $J(\cdot)$ over $\mathcal{U}[0,T]$, that is (\ref{obje-eq-y0}).

We are motivated to take the coefficients in (\ref{state-eqy0}) that satisfy (\ref{prelim-log-coeff}) into account, where an asymmetric risk-sensitive portfolio optimization problem arises due to the $\Gamma$ may not be a scalar matrix.
For simplicity of writing, we write the factor process $\tilde{X}$ determined by (\ref{eq-factor-process}) as $X$ without causing any ambiguity. By the standard theory, it is easy to show that (\ref{eq-factor-process}) admits a unique solution $X(\cdot) \in \bigcap_{p>1} L_{\mathbb{F}}^{p}(\Omega;C([0,T],\mathbb{R}^{n}))$ since it is Gaussian. It follows from (\ref{prelim-log-coeff}) that the controlled BSDE in (\ref{state-eqy0}) can be rewritten as
\begin{equation}
\left\{
\begin{array}
[c]{rl}%
dY(t)= & -\left[  (Z^{\intercal}(t)-u^{\intercal}%
(t)\Sigma)\Gamma(Z(t)-\Sigma^{\intercal}u(t))\right. \\
& +\left.  \frac{1}{2}u^{\intercal}(t)\Sigma\Sigma^{\intercal}%
u(t)-u^{\intercal}(t)(a+AX(t)-r(t)\mathbf{1})-r(t)\right]  dt,\\
& +Z^{\intercal}(t)dW(t),\\
Y(T)= & 0,
\end{array}
\right.
\label{log-BSDE-new}%
\end{equation}
where $\Gamma$ is strictly positive definite. The following assumption is necessary.

\begin{assumption}
\label{assum-Sigma-posit-def} The matrix $\Sigma\Sigma^{\intercal}$ is strictly positive definite. 
\end{assumption}

\begin{remark}
Assumption \ref{assum-Sigma-posit-def} implies that one cannot replicate the
risk structure of one of the $m$ assets by setting up a portfolio of the other
$m-1$ assets. As a result there is no risk-induced arbitrage opportunity on
the market.
\end{remark}


Noting that (\ref{eq-factor-process}) is linear and Gaussian, we are inspired by the heuristic derivation adopted in the subsection 3.3 
to determine the optimal investment strategy.
Let $\Pi(\cdot)\in C([0,T];\mathbb{S}^{n \times n})$ be the unique solution to the Riccati differential equation
\begin{equation}
\left\{
\begin{array}
[c]{rl}%
d\Pi(t)= & -\left[  (B^{\intercal}-A^{\intercal}\Theta^{-1}\Xi)\Pi
(t)-\Pi(t)(B-\Xi^{\intercal}\Theta^{-1}A)\right.  \\
& +\left.  \Pi(t)(\Psi-\Xi^{\intercal}\Theta^{-1}\Xi)\Pi(t)-A^{\intercal
}\Theta^{-1}A\right]  dt,\text{ \ }t\in\lbrack0,T],\\
\Pi(T)= & 0,
\end{array}
\right.  \label{eq-Pi-ode}%
\end{equation}
and let $\varphi(\cdot) \in C([0,T];\mathbb{R}^{n})$ be the unique solution to the
linear ordinary differential equation
\begin{equation}
\left\{
\begin{array}
[c]{rl}%
d\varphi(t)= & -\left\{  [B^{\intercal}-\Pi(t)(\Psi-\Xi^{\intercal}\Theta
^{-1}\Xi)-A^{\intercal}\Theta^{-1}\Xi]\varphi(t)\right.  \\
& +\left.  \Pi(t)[b-\Xi^{\intercal}\Theta^{-1}(a-r(t)\mathbf{1})]+A^{\intercal
}\Theta^{-1}(a-r(t)\mathbf{1})\right\}  dt,\\
\varphi(T)= & 0,
\end{array}
\right.  \label{eq-phi-ode}%
\end{equation}
where
\begin{equation}
\label{def-Theta-Xi-Psi}%
\begin{array}
[c]{ccc}%
\Theta=\Sigma(2\Gamma+\mathrm{I}_{d\times d})\Sigma^{\intercal}, & \Xi
=2\Sigma\Gamma\Lambda^{\intercal}, & \Psi=2\Lambda\Gamma\Lambda^{\intercal}.
\end{array}
\end{equation}
Obviously $\Theta\in\mathbb{S}^{n\times n}$ and $\Theta\geq\Sigma\Sigma^{\intercal} >0$ so it is invertible.

\begin{lemma}
	\label{lem-Riccati-Pi-wellpose} If $\Psi-\Xi^{\intercal}\Theta^{-1}\Xi> 0$,
	then (\ref{eq-Pi-ode}) admits a unique solution $\Pi(t) \geq0, t \in[0,T]$ and
$\left\Vert \Pi\right\Vert _{\infty} \leq B_{\Pi}$, where $B_{\Pi
	}:=\exp\left\{  2\left(  \left\vert B \right\vert + \left\vert \Xi\right\vert
	\left\vert \Theta^{-1} \right\vert \left\vert A \right\vert \right) T \right\}
	\left\vert A \right\Vert ^{2}\left\vert \Theta^{-1} \right\vert T$.
\end{lemma}

\begin{proof}
	Since $\Pi(T)= 0$ and $A^{\intercal}\Theta^{-1}A \geq 0$, according to Theorem 7.5 in \cite{Yong-Zhou}, (\ref{eq-Pi-ode}) admits a unique solution $\Pi(\cdot)\in C([0,T];\mathbb{S}^{n \times n})$ such that $\Pi(t) \geq 0; t \in [0,T]$. Similar to the proof of Lemma \ref{lem-Riccati-wellpose-comparison}, we can deduce $\left\Vert \Pi \right\vert_{\infty} \leq B_{\Pi}$.
\end{proof}

\begin{remark}
	When $\Gamma= \frac{\theta}{4}\mathrm{I}_{d \times d}$, then
	\[%
	\begin{array}
	[c]{rl}%
	\Psi-\Xi^{\intercal}\Theta^{-1}\Xi= & \frac{\theta}{2}\Lambda\left[  \mathrm{I}_{d\times
		d}-\frac{\theta}{\theta+2}\Sigma^{\intercal}\left(  \Sigma\Sigma^{\intercal
	}\right)  ^{-1}\Sigma\right]  \Lambda^{\intercal}\\
	= & \frac{\theta}{2}\Lambda\left[  \left(  \mathrm{I}_{d\times d}-\Sigma^{\intercal
	}\left(  \Sigma\Sigma^{\intercal}\right)  ^{-1}\Sigma\right)  +\frac{1}%
	{\theta+2}\Sigma^{\intercal}\left(  \Sigma\Sigma^{\intercal}\right)
	^{-1}\Sigma\right]  \Lambda^{\intercal}.
	\end{array}
	\]
	Note that $\Sigma^{\intercal}\left(  \Sigma\Sigma^{\intercal}\right)
	^{-1}\Sigma$ is the projection on the column space of $\Sigma$ and therefore
	$\mathrm{I}_{d\times d}-\Sigma^{\intercal}\left(  \Sigma\Sigma^{\intercal
	}\right)  ^{-1}\Sigma$ is an orthogonal projection. As $\theta>0$ the term
\[	
\frac{\theta}{2}\Lambda\left[  \mathrm{I}_{d\times d}-\frac{\theta}{\theta+2}%
	\Sigma^{\intercal}\left(  \Sigma\Sigma^{\intercal}\right)  ^{-1}\Sigma\right]
	\Lambda^{\intercal}>0,
\]
 so (\ref{eq-Pi-ode}) naturally admits a unique
	solution $\Pi(t) \geq0$ defined for all $t \in[0,T]$.
\end{remark}

\begin{definition}
\label{def-log-admiss-control}
An investment strategy $u(\cdot) \in L_{\mathbb{F}}^{2}([0,T];\mathbb{R}^{m})$ is called admissible if it satisfies the following conditions.
\begin{enumerate}
  \item[(i)] The BSDE in (\ref{log-BSDE-new}) admits a minimal solution $$(Y(\cdot),Z(\cdot))\in L_{\mathbb{F}}^{2}(\Omega;C([0,T],\mathbb{R}))\times L_{\mathbb{F}}^{2}([0,T];\mathbb{R}^{d}).$$

  \item[(ii)] For any solution $(Y^{\prime}(\cdot),Z^{\prime}(\cdot))\in L_{\mathbb{F}}^{2}(\Omega;C([0,T],\mathbb{R}))\times L_{\mathbb{F}}^{2}([0,T];\mathbb{R}^{d})$ of the BSDE in (\ref{log-BSDE-new}),
      $$
      \mathbb{E}_{\bar{\mathbb{P}}}\left[ \left(\int_{0}^{T} \left\vert Z^{\prime}(s) \right\vert^{2} ds\right)^{\frac{1}{2}}\right] < +\infty,
      $$
      where $\mathbb{E}_{\bar{\mathbb{P}}}\left[ \cdot \right]$ is the mathematical expectation corresponding to the reference probability $\bar{\mathbb{P}}$ defined by
      \begin{equation}%
\begin{array}
[c]{rl}%
d\mathbb{\bar{P}}:= & \exp\left\{  -2\int_{0}^{T}\chi^{\intercal
}(s)dW(s)-2\int_{0}^{T}\left\vert \chi(s)\right\vert ^{2}ds\right\}
d\mathbb{P},
\end{array}
\label{log-Radon-Nikodym}%
\end{equation}
where $\Pi(\cdot)$, $\varphi(\cdot)$ is the unique solution to (\ref{eq-Pi-ode}), (\ref{eq-phi-ode}) respectively, and
   \begin{equation*}%
\begin{array}[c]{rl}
\chi(s)=&\Gamma\left\{  \left[  \Lambda^{\intercal}\Pi(s)+\Sigma^{\intercal
}\Theta^{-1}\left(  A-\Xi\Pi(s)\right)  \right]  X(s)+\left(  \Lambda
^{\intercal}-\Sigma^{\intercal}\Theta^{-1}\Xi\right)  \varphi(s) \right.\\
&\left.+\Sigma
^{\intercal}\Theta^{-1}\left(  a-r(s)\mathbf{1}\right)  \right\}  .
  \end{array}
   \end{equation*}%

\end{enumerate}
\end{definition}
The set of all admissible strategies will be denoted by $\mathcal{U}_{PO}[0,T]$.

\begin{remark}
Since $X(\cdot)$ is Gaussian and the Girsanov exponential with Gaussian integrand is an exponential martingale, the stochastic exponential in (\ref{log-Radon-Nikodym}) is a Radon-Nikodym derivative.
\end{remark}

The cost functional is defined by
\begin{equation}
J(u(\cdot)):= Y(0), \ u(\cdot) \in \mathcal{U}_{PO}[0,T].
\label{obje-log-recursive}%
\end{equation}
The objective is to find $\bar{u}(\cdot) \in \mathcal{U}_{PO}[0,T]$ (if it ever exists) such that
\[
J(\bar{u}(\cdot))=\inf\limits_{u(\cdot)\in\ \mathcal{U}_{PO}[0,T]}J(u(\cdot)).
\]

\begin{theorem}
\label{thm-log-optimal}
Let Assumption \ref{assum-Sigma-posit-def} hold. Assume $\Pi(\cdot) \in C([0,T];\mathbb{S}^{n \times n})$ uniquely solves (\ref{eq-Pi-ode}) and $\varphi(\cdot) \in C([0,T];\mathbb{R}^{n})$ uniquely solves (\ref{eq-phi-ode}). Then the state feedback strategy
\begin{equation}
\label{log-feedback-control}
\bar{u}(t):=\Theta^{-1}[(A-\Xi\Pi(t))X(t)-\Xi\varphi(t)+(a-r(t)\mathbf{1})], \ \ t \in [0,T].
\end{equation}
belongs to $\mathcal{U}_{PO}[0,T]$ and is optimal for the problem (\ref{log-BSDE-new}), (\ref{obje-log-recursive}). The corresponding optimal value of the objective function is
\[
J(\bar{u}(\cdot))=-\frac{1}{2}x_{0}^{\intercal}\Pi
(0)x_{0}-\varphi^{\intercal} (0)x_{0}-\kappa(0),
\]
where the time dependent coefficient $\kappa(\cdot) \in C([0,T];\mathbb{R})$ is defined as $\kappa(t)=\int_{t}^{T}l(s)ds,$ $ t \in[0,T]$ with
\begin{equation}%
\begin{array}
[c]{rl}%
l(t)= & -\frac{1}{2}\left[  \mathrm{tr}\left\{  \Lambda\Lambda^{\intercal}%
\Pi(t)\right\}  +2r(t)+2b^{\intercal}\varphi(t)-\varphi^{\intercal}%
(t)(\Psi-\Xi^{\intercal}\Theta^{-1}\Xi)\varphi(t)\right. \\
& -\left.  2\varphi^{\intercal}(t)\Xi^{\intercal}\Theta^{-1}(a-r(t)\mathbf{1}%
)+(a-r(t)\mathbf{1})^{\intercal}\Theta^{-1}(a-r(t)\mathbf{1})\right]  .
\end{array}
\label{eq-kappa-ode}%
\end{equation}

\end{theorem}

\begin{proof}
We first show that $\bar{u}(\cdot) \in \mathcal{U}_{PO}[0,T]$. Due to (\ref{log-feedback-control}), it can be verified that $\bar{u}(\cdot) \in L_{\mathbb{F}}^{2}([0,T];\mathbb{R}^{m})$. Applying It\^{o}'s lemma to $\frac{1}{2}X^{\intercal}(t)\Pi(t)X(t)$, $\varphi^{\intercal}(t)X(t)$ respectively, and from (\ref{eq-kappa-ode}), we get
\begin{equation*}%
\begin{array}
[c]{rl}
& \frac{1}{2}X^{\intercal}(t)\Pi(t)X(t)+\varphi^{\intercal
}(t)X(t)+\kappa(t)\\
& =\int_{t}^{T}\left[  -(X^{\intercal}(s)\Pi(s)+\varphi^{\intercal
}(s))\Lambda\Gamma\Lambda^{\intercal}(\Pi(s)X(s)+\varphi
(s))+\frac{1}{2}\bar{u}^{\intercal}(s)\Theta\bar{u}(s)+r(s)\right]  ds\\
& \ \ -\int_{t}^{T}(X^{\intercal}(s)\Pi(s)+\varphi^{\intercal
}(s))\Lambda dW(s)\\
& =\int_{t}^{T}\left\{  -\left[  (X^{\intercal}(s)\Pi(s)+\varphi
^{\intercal}(s))\Lambda+\bar{u}^{\intercal}(s)\Sigma\right]  \Gamma\left[
\Lambda^{\intercal}(\Pi(s)X(s)+\varphi(s))+\Sigma^{\intercal}\bar
{u}(s)\right]  \right.  \\
& \ \ -\left.  \frac{1}{2}\bar{u}^{\intercal}(s)\Sigma\Sigma^{\intercal}%
\bar{u}(s)+\bar{u}^{\intercal}(s)(a+AX(s)-r(s)\mathbf{1}%
)+r(s)\right\}  ds\\
& \ \ -\int_{t}^{T}(X^{\intercal}(s)\Pi(s)+\varphi^{\intercal
}(s))\Lambda dW(s).
\end{array}
\end{equation*}
Therefore, when $u(\cdot)=\bar{u}(\cdot)$, the BSDE in (\ref{log-BSDE-new}) admits a solution
\begin{equation}%
\begin{array}
[c]{rl}%
\bar{Y}(t)= & -\frac{1}{2}X^{\intercal}(t)\Pi(t)X%
(t)-\varphi^{\intercal}(t)X(t)-\kappa(t),\\
\bar{Z}(t)= & -\Lambda^{\intercal}(\Pi(t)X(t)+\varphi(t))
\end{array}
\label{log-relation-optimal-solution}%
\end{equation}
in $L_{\mathbb{F}}^{2}(\Omega;C([0,T],\mathbb{R}))\times L_{\mathbb{F}}^{2}([0,T];\mathbb{R}^{d})$ as $X(\cdot) \in L_{\mathbb{F}}^{4}(\Omega;C([0,T];\mathbb{R}^{n}))$. Moreover, it can be checked that $X(\cdot)$ is also Gaussian under $\bar{\mathbb{P}}$ so we have $\mathbb{E}_{\bar{\mathbb{P}}}\left[ \left(\int_{0}^{T} \left\vert \bar{Z}(s) \right\vert ^{2} ds\right)^{\frac{1}{2}}\right] < +\infty$, which verifies (ii) in Definition \ref{def-log-admiss-control}. We prove $(\bar{Y}(\cdot),\bar{Z}(\cdot))$ is minimal. If $(\bar{Y}^{\prime}(\cdot),\bar{Z}^{\prime}(\cdot)) \in L_{\mathbb{F}}^{2}(\Omega;C([0,T],\mathbb{R}))\times L_{\mathbb{F}}^{2}([0,T];\mathbb{R}^{d})$ is another solution corresponding to $\bar{u}(\cdot)$, then we have for any $t \in [0,T]$,
\[%
\begin{array}
[c]{rl}%
\bar{Y}^{\prime}(t)-\bar{Y}(t)\geq & \int_{t}^{T}2\left[  \bar{Z}^{\prime
}(s)-\bar{Z}(s)\right]  ^{\intercal}\Gamma\left[  \bar{Z}(s)-\Sigma
^{\intercal}\bar{u}(s)\right]  ds\\
& -\int_{t}^{T}\left[  \bar{Z}^{\prime}(s)-\bar{Z}(s)\right]  ^{\intercal}dW(s).
\end{array}
\]
From (\ref{log-Radon-Nikodym}), (\ref{log-feedback-control}), and (\ref{log-relation-optimal-solution}), applying Girsanov's theorem yields that
\begin{equation}
\label{log-difference-barYp-barY}
\bar{Y}^{\prime}(t)-\bar{Y}(t)\geq -\int_{t}^{T}\left[  \bar{Z}^{\prime}(s)-\bar{Z}(s)\right]  ^{\intercal}d\bar{W}(s),
\end{equation}
where
\[
\bar{W}(t)=W(t)+2\int_{0}^{t}\Gamma\left[  \bar{Z}(s)-\Sigma^{\intercal}\bar{u}(s)\right]  ds
\]
is an $d$-dimensional Brownian motion under $\mathbb{\bar{P}}$. As $\mathbb{E}_{\bar{\mathbb{P}}}\left[ \left(\int_{0}^{T} \left\vert \bar{Z}^{\prime}(s) \right\vert ^{2} ds\right)^{\frac{1}{2}}\right] < +\infty$ due to (ii) in Definition \ref{def-log-admiss-control}, we deduce $\bar{Y}^{\prime}(t) \geq \bar{Y}(t)$, $t \in [0,T]$, $\bar{\mathbb{P}}$-a.s. (of course, $\mathbb{P}$-a.s.) by taking the conditional expectation $\mathbb{E}\left[ \cdot \mid \mathcal{F}_{t} \right]$ in both sides of (\ref{log-difference-barYp-barY}).
Hence $\bar{u}(\cdot) \in \mathcal{U}_{PO}[0,T]$ and we obtain
\begin{equation}
\label{log-barY-eq-Jbaru}
J(\bar{u}(\cdot)) = \bar{Y}(0).
\end{equation}

Now we prove the optimality of $\bar{u}(\cdot)$. Taking any admissible strategy $u(\cdot)\in\mathcal{U}_{PO}[0,T]$ and the corresponding minimal solution $(Y(\cdot),Z(\cdot)) \in L_{\mathbb{F}}^{2}(\Omega;C([0,T],\mathbb{R}))\times L_{\mathbb{F}}^{2}([0,T];\mathbb{R}^{d})$, we obtain
\begin{equation}%
\begin{array}
[c]{rl}
& Y(0)-\bar{Y}(0)\\
\geq & \int_{0}^{T}\left\{  2\left[  Z(s)-\bar{Z}(s)\right]  ^{\intercal
}\Gamma\left[  \bar{Z}(s)-\Sigma^{\intercal}\bar{u}(s)\right]  \right.  \\
& +\left.  \left[  u(s)-\bar{u}(s)\right]  ^{\intercal}\left[  \Theta\bar
{u}(s)-2\Sigma\Gamma\bar{Z}(s)-\left(  a+AX(s)-r(s)\mathbf{1}\right)
\right]  \right\}  ds\\
& -\int_{0}^{T}\left[  Z(s)-\bar{Z}(s)\right]  ^{\intercal}dW(s).
\end{array}
\label{eq-opti-difference-app}%
\end{equation}
From (\ref{log-feedback-control}) and (\ref{eq-opti-difference-app}), we get
\[
\Theta\bar{u}(s)-2\Sigma\Gamma\bar{Z}(s)-\left(  a+AX(s)-r(s)\mathbf{1}\right)=0
\]
and therefore
\begin{equation}
Y(0)-\bar{Y}(0)\geq-\int_{0}^{T}\left[  Z(s)-\bar
{Z}(s)\right]  ^{\intercal}d\bar{W}(s)
\label{deffer-geq-local-martingale}%
\end{equation}
by Girsanov's theorem. Since $\bar{u}(\cdot) \in \mathcal{U}_{PO}[0,T]$ and $\mathbb{E}_{\bar{\mathbb{P}}}\left[ \left(\int_{0}^{T} \left\vert Z(s) \right\vert^{2} ds\right)^{\frac{1}{2}}\right]  <+\infty$ according to (ii) in Definition \ref{def-log-admiss-control}, the right-hand side in (\ref{deffer-geq-local-martingale}) is a true martingale admitting a mean zero under $\mathbb{\bar{P}}$. Taking $\mathbb{E}_{\mathbb{\bar{P}}}\left[  \cdot \right]$ in both sides of (\ref{deffer-geq-local-martingale}), we have $Y(0)-\bar{Y}(0) \geq 0$ which means that
\begin{equation}
\label{log-barY-leq-Jbaru}
J(u(\cdot)) \geq J(\bar{u}(\cdot)).
\end{equation}
The optimality of $\bar{u}(\cdot)$ follows from (\ref{log-barY-leq-Jbaru}) and the arbitrariness of $u(\cdot)$ chosen from $\mathcal{U}_{PO}[0,T]$. Finally, combining the relationship (\ref{log-relation-optimal-solution}) with (\ref{log-barY-eq-Jbaru}), we deduce the optimal value
\[
J(\bar{u}(\cdot)) = \bar{Y}(0)=-\frac{1}{2}x^{\intercal}_{0}\Pi(0)x_{0}-\varphi^{\intercal}(0)x_{0}-\kappa(0),
\]
which accomplishes the proof.
\end{proof}

\begin{remark}
As the original problem is to maximize the expected growth rate $-J(\cdot)$
over $\mathcal{U}_{PO}[0,T]$, it follows from Theorem \ref{thm-log-optimal}
that the optimal growth rate is $\frac{1}{2}x_{0}^{\intercal}\Pi(0)x_{0}+\varphi
^{\intercal}(0)x_{0}+\kappa(0)$. Furthermore, when $\Gamma= \frac{\theta}{4}
\mathrm{I}_{d \times d}$ for some given $\theta>0$, this result degenerates
into the same one as Theorem 2.1 in \cite{Kuroda-Nagai} since it can
be verified that $\left(  \Pi(\cdot), \varphi(\cdot), \kappa(\cdot) \right) $
satisfy (2.16), (2.17), (2.18) respectively on pages 316-317 in
\cite{Kuroda-Nagai}.
\end{remark}

\section*{Funding}

The work of Mingshang Hu was partially supported by the National Natural Science Foundation of China  (Grant No. 12326603, 11671231).

The work of Shaolin Ji was partially supported by the National Key R\&D Program of China (Grant No. 2023YFA1008701) and the Key Project of the National Natural Science Foundation of China (Grant No. 12431017).

The work of Rundong Xu was partially supported by China Postdoctoral Science Foundation (Grant No.2024M760481) and Shanghai Postdoctoral Excellence Program (Grant No.2023201).

The work of Xiaole Xue was partially supported by National Natural Science Foundation of China (Grant No.12471420;12001316).

\end{document}